\documentclass{article}

\usepackage{microtype}
\usepackage{graphicx}
\usepackage{subfig}
\usepackage{booktabs} % for professional tables
\usepackage{amsfonts}
\usepackage{bm} 
\usepackage{dsfont}
\usepackage{amsthm}
\usepackage{amsmath, amssymb}

\usepackage{hyperref}

\newtheorem{defn}{Definition}[section]
\newtheorem{thm}{Theorem}[section]
\newtheorem{lem}{Lemma}[section]

\newtheorem{assump}{Assumption}[section]

\newcommand{\R}{\mathbb{R}}

\newcommand{\bX}{\mathds{X}}
\newcommand{\bY}{\mathds{Y}}

\newcommand{\eps}{\epsilon}
\newcommand{\ud}{\,\mathrm{d}}
\newcommand{\ovd}{\frac{1}{\delta}}

\newcommand{\EE}{\mathbb{E}}

\newcommand{\PP}{\mathbb{P}}
\newcommand{\FF}{\mathbb{F}}
\newcommand{\RR}{\mathbb{R}}
\newcommand{\mcF}{\mathcal{F}}

\newcommand{\mcP}{\mathcal{P}}
\newcommand{\mcL}{\mathcal{L}}
\newcommand{\mc}[1]{\mathcal{#1}}
\newcommand{\half}{\frac{1}{2}}

\newcommand{\transpose}{^{\operatorname{T}}}

\DeclareMathOperator*{\argmin}{arg\,min}

\numberwithin{equation}{section}

\usepackage[accepted]{icml2021}

\icmltitlerunning{Sig-DFP for MFG with Common Noise}

\begin{document}

\twocolumn[
\icmltitle{Signatured Deep Fictitious Play for Mean Field Games with Common Noise}



\icmlsetsymbol{equal}{*}

\begin{icmlauthorlist}
\icmlauthor{Ming Min}{equal,pstat}
\icmlauthor{Ruimeng Hu}{equal,pstat,math}
\end{icmlauthorlist}

\icmlaffiliation{pstat}{Department of Statistics and Applied Probability, University of California, Santa Barbara, CA 93106-3110, USA}
\icmlaffiliation{math}{Department of Mathematics, University of California, Santa Barbara, CA 93106-3080, USA}

\icmlcorrespondingauthor{Ruimeng Hu}{rhu@ucsb.edu}

\icmlkeywords{Mean field games with common noise, rough path theory, signature, deep fictitious play}

\vskip 0.3in
]



\printAffiliationsAndNotice{\icmlEqualContribution} % otherwise use the standard text.

\begin{abstract}

Existing deep learning methods for solving mean-field games (MFGs) with common noise fix the sampling common noise paths and then solve the corresponding MFGs. This leads to a nested-loop structure with millions of simulations of common noise paths in order to produce accurate solutions, which results in prohibitive computational cost and limits the applications to a large extent. In this paper, based on the rough path theory, we propose a novel single-loop algorithm, named signatured deep fictitious play, by which we can work with the unfixed common noise setup to avoid the nested-loop structure and reduce the computational complexity significantly. The proposed algorithm can accurately capture the effect of common uncertainty changes on mean-field equilibria without further training of neural networks, as previously needed in the existing machine learning algorithms. The efficiency is supported by three applications, including linear-quadratic MFGs, mean-field portfolio game, and mean-field game of optimal consumption and investment. Overall, we provide a new point of view from the rough path theory to solve MFGs with common noise with significantly improved efficiency and an extensive range of applications. In addition, we report the first deep learning work to deal with extended MFGs (a mean-field interaction via both the states and controls) with common noise. 
\end{abstract}

\section{Introduction}

Stochastic differential games study the strategic interaction of rational decision-makers in an uncertain dynamical system, and have been widely applied to many areas, including social science, system science, and computer science. For realistic models, the problem usually lacks tractability and needs numerical methods. With a large number of players resulting in high-dimensional problems, conventional algorithms soon lose efficiency and one may resort to recently developed machine learning tools \cite{Hu2:19,HaHu:19,han2020convergence}. On the other hand, one could utilize its limiting mean-field version, mean-field games (MFGs), to approximate the $n$-player game for large $n$ ({\it e.g.}, \citet{han2021GMMD}). Introduced independently in \citet{huang2006large,lasry2007mean}, MFGs study the decision making problem of a continuum of agents, aiming to provide asymptotic analysis of the finite player model in which players interact through their empirical distribution. In an MFG, each agent is infinitesimal, whose decision can not affect the population law. Therefore, the problem can be solved by focusing on the optimal decision of a representative agent in response to the average behavior of the entire population and a fixed-point problem ({\it cf.} equation \eqref{def:fixedpoint}). The MFG model has inspired tremendous applications, not only in finance and economics, such as system risk \cite{carmona2015mean}, high-frequency trading \cite{lachapelle2016efficiency}  and crowd trading \cite{cardaliaguet2017mean}, but also to population dynamics \cite{achdou2017mean,djehiche2017mean,achdou2019mean} and sanitary vaccination \cite{hubert2018nash,elie2020contact}, to list a few. For a systematical introduction of MFGs, see \citet{caines2015mean, carmona2018probabilistic,carmona2018probabilistic2}.

In MFGs, the random shocks to the dynamical system can be from two sources: idiosyncratic to the individual players and common to all players, \emph{i.e.}, decision-makers face correlated randomness.  While MFGs were initially introduced with only idiosyncratic noise as seen in most of the literature, games with common noise, referred to as \emph{MFGs with common noise}, have attracted significant attention recently \cite{lacker2015translation,carmona2016mean,ahuja2016wellposedness,graber2016linear}. The inclusion of common noise is natural in many contexts, such as multi-agent trading in a common stock market, or systemic risk induced through inter-bank lending/borrowing. In reality, players make decisions in a common environment ({\it e.g.}, trade in the same stock market). Therefore, their states are subject to correlated random shocks, which can be modeled by individual noises and a common noise. In this modeling, observing the state dynamics will be sufficient, and one does not need to observe the noises. These applications make it crucial to develop efficient and accurate algorithms for computing MFGs with common noise.

Theoretically, MFGs with common noise can be formulated as an infinite-dimensional master equation, which is the type of second-order nonlinear Hamilton-Jacobi-Bellman equation involving derivatives with respect to a probability measure. Therefore, direct simulation is infeasible due to the difficulty of discretizing the probability space. An alternative way of solving MFGs with common noise is to formulate it into a stochastic Fokker-Planck/Hamilton-Jacobi-Bellman system, which has a complicated form with common noise, forward-backward coupling, and second-order differential operators. The third kind of approaches turns it into forward backward stochastic differential equations (FBSDE) of McKean-Vlasov type ({\it cf.} \citet[Chapter 2]{carmona2018probabilistic2}), which in general requires convexity of the Hamiltonian. For all three approaches, the common assumption is the monotonicity condition that ensures uniqueness. Regarding simulation, existing deep learning methods fix the sampling common noise paths and then solve the corresponding MFGs, which leads to a nested-loop structure with millions of simulations of common noise paths to produce accurate 
predictions for unseen common shock realizations. Then the computational cost becomes prohibitive and limits the applications to a large extent.

%This approach requires the Hamiltonian to be convex, which is not always true. 
%Another method is to solve the stochastic partial differential equations (SPDE) \cite{kolokoltsov2015mean}, which is complicated due to the appearance of common noise and second order differential operators. 

In this paper, we solve MFGs with common noise by directly parameterizing the optimal control using deep neural networks in spirit of \cite{han2016deep}, and conducting a global optimization. We integrate the signature from rough path theory, and fictitious play from game theory for efficiency and accuracy, and term the algorithm \emph{Signatured Deep Fictitious Play} (Sig-DFP). The proposed algorithm avoids solving the three aforementioned complicated equations (master equation, Stochastic FP/HJB, FBSDE) and does not have uniqueness issues.

{\bf Contribution.} We design a novel efficient single-loop deep learning algorithm, Sig-DFP, for solving MFGs with common noise by integrating fictitious play \cite{brown1949some} and Signature \cite{LyonsTerryJ2007DEDb} from rough path theory. {To our best knowledge, this is the first work focusing on the common noise setting, which can address heterogeneous MFGs and heterogeneous extended MFGs, both with common noise.} % which improves the time and space complexity by a factor of $\mc{O}(N)$ with $N$ as the sample size. 
  
We prove that the Sig-DFP algorithm can reach mean-field equilibria as both the depth $M$ of the truncated signature and the stage $n$ of the fictitious play approaching infinity, subject to the universal approximation of neural networks.
   We demonstrate its convergence superiority on three benchmark examples, including homogeneous MFGs, heterogeneous MFGs, and heterogeneous extended MFGs, all with common noise, and with assumptions even beyond the technical requirements in the theorems. Moreover, the algorithm has the following advantages: 

1. Temporal and spacial complexity are $\mc{O}(NLp+Np^2)$ and $\mc{O}(NLp)$, compared to $\mc{O}(N^2 L)$ (for both time and space) in existing machine learning algorithms, with $N$ as the sample size, $L$ as the time discretization size, $p = \mc{O}(n_0^M)$, $n_0$ as the dimension of common noise.

2. Easy to apply the fictitious play strategy: only need to average over linear functionals with $\mc{O}(1)$ complexity.

{\bf Related Literature.} After MFGs firstly introduced by \citet{huang2006large} and \citet{lasry2007mean} under the setting of a continuum of homogeneous players but without common noise, it has been extended to many applicable settings, {\it e.g.}, heterogeneous players games \cite{lacker2018mean,lacker2020many} and major-minor players games \cite{huang2010large,nourian2013,carmona2016probabilistic}. A recent line of work studies MFGs with common noise \cite{carmona2015mean,bensoussan2015master,ahuja2016wellposedness,cardaliaguet2019master}. Despite its theoretical progress and importance for applications, efficient numerical algorithms focusing on common noise settings are still missing. Our work will fill this gap by integrating machine learning tools with learning procedures from game theory and signature from rough path theory. 

%After  MFGs  firstly  introduced  byLasry-Lions (Lasry & Lions, 2007) and Huang-Malham ́e-Caines (Huang et al., 2006) under the setting of a continuumof homogeneous players and without common noise, it soonattracted many attentions, and has been generalized to thesetting with heteogeneous players (Lacker & Zariphopoulou,2018; Lacker & Soret, 2020), or with one or more majorplayers with minors (Huang, 2010; Nourian & Caines, 2013;Carmona et al., 2016b

%book by carmona and delarue. Later generator to heterogeneous players, major player with minors, and on mean field game with common noise (widely studied in the literature and crucial for applications), application in many fields (finance, 

%approaches (HJB + FP, FBSDE, master equation)
Fictitious play was firstly proposed in \citet{brown1949some,brown1951iterative} for normal-form games, as a learning procedure for finding Nash equilibria. It has been widely used in the Economic literature, and adapted to MFGs \cite{cardaliaguet2015learning,briani2018stable} and finite-player stochastic differential games \cite{Hu2:19,HaHu:19,han2020convergence,xuan2020optimal}.

Using machine learning to solve MFGs has also been considered, for both model-based setting \cite{carmona2019convergence,ruthotto2020machine,lin2020apac} and model-free reinforcement learning setting \cite{guo2019learning,tiwari2019reinforcement,angiuli2020unified,EliePerolatLauriereGeistPietquin-2019_AFP-MFG}, most of which did not consider common noise. Existing machine learning methods for MFGs with common noise were studied in \citet{perrin2020fictitious}, which have a nested-loop structure and require millions of simulations of common noise paths to produce accurate predictions for unseen common shock realizations.

%also use fictitious + machine learning in MFGs. 

%machine learning algorithms in mfg, model-based (carmona-Lauriere), model-free(Lauriere, xin guo), mostly without common noise, 
The signature in rough path theory has been recently applied to machine learning as a feature map for sequential data. For example, \citet{JMLR:v20:16-314, NEURIPS2019_deepsig,  toth2019bayesian,  min2020convolutional} have used signatures in natural language processing, time series, and handwriting recognition, and \citet{chevyrev2018signature, ni2020conditional} studied the relation between signatures and distributions of sequential data. We refer to \citet{lyons2002system, LyonsTerryJ2007DEDb} for a more detailed introduction of the signature and rough path theory.

%Signature itself is an interesting topic and an efficient tool for machine learning in sequential data, for example \cite{JMLR:v20:16-314}, \cite{min2020convolutional}, \cite{toth2019bayesian} to name a few.
%A detailed introduction of signature and rough paths theory can be found in \cite{lyons2002system}, \cite{LyonsTerryJ2007DEDb}.

%\rh{Please add:
%literature in signature, using signature to solve mfg, deep learning algorithm using signature.
%}

\section{Mean Field Games with Common Noise}

%Mean field games (MFG) study the decision making problem of a continuum of agents, homogeneous or heterogeneous. In an MFG, each agent is infinitesimal, thus the problem can be solved by focusing on the optimal control of a representative agent in response to the average behavior of the entire population. In an MFG with common noise, all agents make decisions subject to some common randomness. 

We first introduce the following notations to precisely define MFGs with common noise. For a fixed time horizon $T$, let $(W_t)_{0 \leq t \leq T}$ and $(B_t)_{0 \leq t \leq T}$ be independent $n$- and $n_0$-dimensional Brownian motions defined on a complete filtered probability space $(\Omega, \mcF, \FF = \{\mcF_t\}_{0\leq t\leq T}, \PP)$. We shall refer $W$ as the \emph{idiosyncratic noise} and $B$ as the \emph{common noise} of the system. Let $\mcF^B_t$ be the filtration generated by $(B_t)_{0 \leq t \leq T}$, and $\mc{P}^p(\RR^d)$ be the collection of probability measures on $\RR^d$ with finite $p^{th}$ moment, {\it i.e.}, $\mu \in \mc{P}^p(\RR^d)$ if
\begin{equation}
    \left(\int_{\RR^d} \|x\|^p \ud \mu(x)\right)^{1/p} < \infty.
\end{equation}
We denote by $\mc{M}([0,T]; \mcP^2(\RR^d))$ the space of continuous $\mcF^B$-adapted stochastic flow of probability measures with the finite second moment, and by $\mc{H}^2([0, T]; \RR^m)$ %(\emph{resp.} $\mc{H}^2_{\mcF^B}([0, T]; \RR^m)$) 
the set of all $\mcF$-progressively measurable $\RR^m$-valued square-integrable processes.

%(\emph{resp.} $\mcF^B$-)

Next, we introduce the concept of MFGs with common noise. Given an initial distribution $\mu_0 \in \mcP^2(\RR^d)$, and a stochastic flow of probability measures $\mu = (\mu_t)_{0 \leq t \leq T} \in \mc{M}([0,T]; \mcP^2(\RR^d))$, we consider the stochastic control %problem
\begin{align}
        &\inf_{(\alpha_t)_{0 \leq t \leq T}} \EE[\int_0^T f(t, X_t, \mu_t, \alpha_t) \ud t + g(X_T, \mu_T)], \label{def:J}\\
        &\text{where }
        \ud X_t = b(t, X_t, \mu_t, \alpha_t)\ud t + \sigma(t, X_t, \mu_t, \alpha_t) \ud W_t \nonumber \\ 
        & \hspace{5.5em} + \sigma^0(t, X_t, \mu_t, \alpha_t) \ud B_t, \label{def:Xt}
\end{align}
with $X_0 \sim \mu_0$. Here the representative agent controls his dynamics $X_t$ through a $\RR^m$-dimensional control process $\alpha_t$, and the drift coefficient $b$, diffusion coefficients $\sigma$ and $\sigma^0$, running cost $f$ and terminal cost $g$ are all measurable functions, with $(b, \sigma, \sigma^0, f): [0,T] \times \RR^d\times \mcP^2(\RR^d) \times  \RR^m \to \RR^d \times \RR^{d \times n} \times \RR^{d \times n_0}\times \RR$, and $g: \RR^d \times \mcP^2(\RR^d) \to \RR$.

Note that since $\mu$ is stochastic, \eqref{def:J}--\eqref{def:Xt} is a control problem with random coefficients.  

%\begin{align}
%    (b, \sigma, \sigma^0, f)&: [0,T] \times \RR^d\times \mcP^2(\RR^d) \times  \RR^m \to \RR^d \times \RR^n \times \RR^{n_0} \times \RR, \\
 %   g&: \RR^d \times \mcP^2(\RR^d) \to \RR.
%\end{align}

\begin{defn}[Mean-field equilibrium]\label{defn:MFG}
The control-distribution flow pair $\alpha^\ast = (\alpha^\ast_t)_{0 \leq t \leq T} \in \mc{H}^2([0, T]; \RR^m)$, $\mu^\ast \in \mc{M}([0,T]; \mcP^2(\RR^d))$ is a mean-field equilibrium to the MFG with common noise, if $\alpha^\ast$ solves \eqref{def:J} given the stochastic measure flow $\mu^\ast$, and the conditional marginal distribution of the optimal path $X_t^{\alpha^\ast}$ given the common noise $B$ coincides with the measure flow $\mu^\ast$:
\begin{equation}\label{def:MFG}
    \mu_t^\ast = \mcL(X_t^{\alpha^\ast} \vert \mcF_t^B),
    %\mu(\cdot) = \PP(X_t^{\alpha^\ast} \in \cdot \vert B).
\end{equation}
where $\mcL(\cdot \vert \mc{F})$ is the conditional law given a filtration $\mc{F}$.
\end{defn}

We remark that, with a continuum of agents, the measure $\mu^\ast$ is not affected by a single agent's choice, and the MFG is a standard control problem plus an additional fixed-point problem. More precisely, denote by $\hat\alpha^\mu$ the optimal control of \eqref{def:J}--\eqref{def:Xt} given the stochastic measure flow $\mu \in \mc{M}([0,T]; \mcP^2(\RR^d))$, then $\mu^\ast$ is a fixed point of  %the relation
\begin{equation}\label{def:fixedpoint}
    \mu_t = \mcL(X_t^{\hat\alpha^\mu}\vert \mcF_t^B).
\end{equation}

{\it MFGs without common noise:} Note that with $\sigma^0 \equiv 0$, \eqref{def:J}--\eqref{def:Xt} is a MFG without common noise, and the flow of measures $\mu_t$ becomes deterministic. 

{\it Extended MFGs:} In extended mean field games, the interactions between the representative agent and the population happen via both the states and controls, thus the functions $(b, \sigma, \sigma^0, f, g)$ can also depend on $\mc{L}(\alpha_t \vert \mcF_t^B)$. 

%\rh{motivation, literature, related to MFG without common noise. }

\section{Fictitious Play and Signatures}

The Signatured Deep Fictitious Play (Sig-DFP) algorithm is built on fictitious play, and propagates conditional distributions $\mu = \{\mu_t\}_{0\le t\le T} \in \mc{M}([0,T]; \mc{P}^2(\RR^d))$ by signatures. This section briefly introduces these two ingredients. 
%We briefly describe the concepts of fictitous play and signagures of paths here. 

%One essential idea of Sig-DFP algorithm is tackle the second part, i.e. 

%{\it Fictitious play} was firstly introduced by \citet{brown1949some,brown1951iterative} for finding Nash equilibria in normal-form games. 

In the learning procedure of {\it fictitious play}, players myopically choose their best responses against the empirical distribution of others' actions at every subsequent stage after arbitrary initial moves. When \citet{cardaliaguet2015learning,cardaliaguet2017mean} extended it to mean-field settings, the empirical distribution of actions is naturally replaced by the average of distribution flows. More precisely, let $\Bar{\mu}^{(0)}\in  \mc{M}([0,T]; \mc{P}^2(\RR^d))$ be the initial guess of $\mu^\ast$ in \eqref{def:MFG}, and consider the following iterative algorithm: (1) take $\Bar{\mu}^{(n-1)}\in\mc{P}^2(\R^d)$ as the given flow of measures in \eqref{def:J}--\eqref{def:Xt} for the $n$-th iteration, and solve the optimal control in \eqref{def:J} denoted by $\alpha^{(n)}$; (2) solve the controlled stochastic differential equation (SDE) \eqref{def:Xt} for $X^{\alpha^{(n)}}$ and then infer the conditional distribution flow $\mu^{(n)} = \mc{L}(X^{\alpha^{(n)}} \vert \mc{F}_t^B)$; (3) average distributions $\Bar{\mu}^{(n)} = \frac{n-1}{n}\Bar{\mu}^{(n-1)} + \frac{1}{n}\mu^{(n)}$ and pass $\Bar{\mu}^{(n)}$ to the next iteration. If $\mu^{(n)}$ converges and the strategy corresponding to the limiting measure flow is admissible, then by construction, it is a fixed-point of \eqref{def:fixedpoint} and thus a mean-field equilibrium. 

{\it Signatures of Paths.} Let $T((\R^d)):=\bigoplus_{k=0}^\infty (\R^d)^{\bigotimes k}$ be the tensor algebra, and denote by $\mathcal{V}^p([0,T], \R^d)$ the space of continuous mappings from $[0,T]$ to $\R^d$ with finite $p$-variation. For a path $x:[0,T]\to \RR^d$, define the $p$-variation %is defined by 
\begin{equation}
    \|x\|_{p} := \left( \sup_{D\subset[0,T]} \sum_{i=0}^{r-1} \|x_{t_{i+1}}-x_{t_i}\|^p \right)^{1/p},
\end{equation} 
where $D \subset [0,T]$ denotes a partition $0 \leq t_0 < t_1 < \ldots < t_r \leq T$. We equip the space $\mathcal{V}^p([0,T], \R^d)$ with the norm $  \|\cdot\|_{\mathcal{V}^p}:=\|\cdot\|_{\infty}+\|\cdot\|_{p}$. % {\it i.e.}, the supremum norm plus $p$-variation.

\begin{defn}[Signature]
Let $X\in \mathcal{V}^p([0,T], \R^d)$ such that the following integral makes sense. The signature of $X$, denoted by $S(X)$, is an element of $T((\R^d))$ defined by $S(X) = (1, X^1, \cdots, X^k \cdots)$ with
\begin{equation}\label{def:signature}
    X^k = \int_{0<t_1<t_2<\cdots<t_k<T} \ud X_{t_1}\otimes\cdots\otimes \ud X_{t_k}.
\end{equation}
We denote by $S^M(X)$ the truncated signature of $X$ of depth $M$, {\it i.e.}, $S^M(X) = (1, X^1, \cdots, X^M)$ and has the dimension $\frac{d^{M+1}-1}{d-1}$.
\end{defn}
Note that when $X$ is a semi-martingale (the case of our problems), equation \eqref{def:signature} is understood in the Stratonovich sense. The following properties of the signature make it an ideal choice for our problem, with more details in  Appendix~\ref{app:signature}. 

%and we shall detail them further in Appendix~\ref{app:signature}. 

1. Signatures characterize paths uniquely up to the tree-like equivalence, and the equivalence is removed if at least one dimension of the path is strictly increasing \cite{boedihardjo2014signature}. Therefore, we shall augment the original path with the time dimension in the algorithm, {\it i.e.}, working with $\hat{X}_t = (t, X_t)$ since $S(\hat{X})$ characterizes paths $\hat{X}$ uniquely.
    
2. Terms in the signature present a factorial decay property, which provides the accuracy of using a few terms in the signature (small $M$) to approximate a path. 
 
3. As a feature map of sequential data, the signature has a universality detailed in the following theorem.

%\begin{itemize}
%    \item Signatures characterize paths uniquely up to tree like equivalence \cite{boedihardjo2014signature}. In the following of this paper, we will augments our original path with additional time dimension, i.e. $\hat{X}_t = (t, X_t)$ and $S(\hat{X})$ characterize path uniquely.
 %   \item Terms in signature enjoys factorial decay property. Therefore, truncation becomes reasonable since most information of a path is included in first several terms in signature.
  %  \item As a feature map of sequential data, signature has universality.
%\end{itemize}

\begin{thm}[Universality, \citet{NEURIPS2019_deepsig}]
    Let $p\ge 1$ and $f: \mathcal{V}^p([0,T], \R^d)\to \R$ be a continuous function in paths. For any compact set $K\subset \mathcal{V}^p([0,T], \R^d)$, if $S(x)$ is a geometric rough path for any $x\in K$, then for any $\eps >0$ there exist $M>0$ and a linear functional $l \in T((\RR^d))^\ast$ such that
    \begin{equation}
        \sup_{x\in K} |f(x) - \langle l, S(x) \rangle| <\epsilon.
    \end{equation}
    \label{thm:sig_universality}
\end{thm}
%\rh{any reference? I only see a statement on the signature $S(x)$ instead of on $S^M(x)$ in [Theorem 4.2, Derivative pricing using signature payoffs], or when the compact set is on the range of signatures of paths in $ \mathcal{V}^p([0,T], \R^d)$ [Theorem 3.1, Learning from the past, predicting the statistics for the future, learning an evolving system]. }

\section{The Sig-DFP Algorithm}

%Before introducing algorithm, we introduce two simple notations here. If $x$ is a path indexed in $[0,T]$, we denote $x:=(x_t)_{0\le t\le T}$ and $x_{s:t}:=(x_{u})_{s\le u\le t}$.

We introduce two shorthand notations: if $x$ is a path indexed by $ t \in [0,T]$, then $x:=(x_t)_{0\le t\le T}$ denotes the whole path and $x_{s:t}:=(x_{u})_{s\le u\le t}$ denotes the path between $s$ and $t$.

\subsection{Propagation of Distribution with Signatures}\label{sec:propsig}
With the presence of common noise, existing algorithms mostly consider a nested-loop structure, with the inner one for idiosyncratic noise $W$ and the outer one for common noise $B$. %Both loops are required to be large in order to perform well for unseen common noise paths and to approximate well the flow of measures $\mu$.
More precisely, if one works with $N$ idiosyncratic Brownian paths $\{W^k\}_{k=1}^N$ and $N$ common Brownian paths $\{B^k\}_{k=1}^N$, then for each $B^j$, one needs to simulate $N$ paths $\{X^{i,j}\}_{i=1}^N$ defined by \eqref{def:Xt} over all idiosyncratic Brownian paths and solve the problem \eqref{def:J} associated to $B^j$. This requires a total of $N^2$ simulations of \eqref{def:Xt}. With a sufficiently large $N$, $\mu_t = \mc{L}(X_t \vert \mc{F}_t^B)$ is approximated well by $\frac{1}{N^2}\sum_{i,j=1}^N \delta_{X_t^{i,j}} \mathds{1}_{\omega^{(0,j)}}$ with $\omega^{0,j} \in \Omega$ corresponding to the trajectory $B^j$. The double summation is of $\mc{O}(N^2)$ which is computationally expensive for large $N$.

%Also, propagation of distribution has to be computed in the same way. For each common noise $B^j$, one averages over $\{X^{i,j}\}_{i=1}^N$ sampled from the common noise to get the conditional distribution flow $\mu^j$. Therefore, the computational cost becomes unaffordable for large $N$. In addition, to obtain a good performance, existing machine learning algorithms require the same common noises to be used for both training and testing, which means one has to take a large $N$ in order to achieve small generalization error.     

%which is inevitable to achieve high accuracy. Another disadvantage is that the same common noises have to be used for both training and testing tasks.

We shall address the aforementioned numerical difficulties by signatures. The key idea is to approximate $\mu_t$ by
\begin{align}
  &\mu_t \equiv \mc{L}(X_t \vert \mc{F}_t^B ) = \mc{L}(X_t \vert S(\hat B_t)) \approx \mc{L}(X_t \vert S^M(\hat B_t)), \nonumber\\
  &\text{with } \hat B_t = (t, B_t), \label{eq:propbysig}
\end{align}
where the equal sign comes from the unique characterization of signatures $S(\hat B)$ to the paths $B_{0:t}$, and the approximation is accurate for large $M$ due to the factorial decay property of the signature. The last term is then computed by machine learning methods, {\it e.g.}, by Generative Adversarial Networks (GANs). In addition, if the agents interact via some population average subject to common noise: $\mu_t = \EE[\iota(X_t)\vert \mc{F}_t^B]$, the approximation in \eqref{eq:propbysig} can be arbitrarily close to the true measure flow for sufficiently large $M$. The following lemma gives a precise statement.

%derive $N$ conditional distribution flows $\{\mu_t^i\}_{i=1}^N$. Since $\mu_t$ is the distribution conditioned on common noise $B_{0:t}$, it can be understood as a function of common noises, {\it i.e.}, $\mu_t = \mu(B_{0:t})$, and we have the following lemma.

\begin{lem}\label{lemma:propbysig}
Suppose $\mu_t=\EE[\iota(X_t)\vert \mc{F}_t^B]$ where $\iota: \RR^d \to \RR$ is a measurable function.  View $\mu_t$ as 
$\mu(t, B_{0:t})$ with $\mu: \mathcal{V}^p([0,T], \R^{n_0+1})\to \R$ continuous for some $p\in(2,3)$, and let $K\subset \mathcal{V}^p([0,T], \R^{n_0+1})$ be a compact set, then for any $\eps >0$, there exist a positive integer $M$ and a linear functional $l \in T((\RR^{n_0+1}))^\ast$, such that
\begin{equation}\label{lem:PropDistBySig}
    \sup_{t\in [0,T]}\sup_{\hat{B}\in K} |\mu_t - \langle l, S^M(\hat{B}_{0:t}) \rangle |<\epsilon.
\end{equation}
\end{lem}
\begin{proof}
See Appendix~\ref{app:signature} for details due to the page limit.
\end{proof}

%To ease the notation and be consistent with our numerical examples, we shall illustrate the algorithm where the interaction is via the population mean subject to common noise, \emph{i.e.}, $\mu_t = \EE[\iota(X_t) \vert \mc{F}_t^B]$. Viewing $\mu_t = \mu(t, B_{0:t})$ as a continuous map of $B_{0:t}$, we have the following lemma.

With all the above preparations, we now explain how the approximation to $\mu = \{\mu_t\}_{0 \leq t \leq T}$ using signatures is implemented. Given $N$ pairs of idiosyncratic and common Brownian paths $(W^i, B^i)$ and assume $\alpha_t$ in \eqref{def:Xt} is already obtained (which will be explained in Section~\ref{sec:Sig-DFP}), we first sample the optimized state processes $(X^i_t)_{0\le t\le T}$, producing $N$ samples $\{X^i\}_{i=1}^N$. Then the linear functional $l$ in Lemma \ref{lemma:propbysig} is approximated by implementing linear regressions on $\{S^M(\hat{B}^i_{0:t})\}_{i=1}^N$ with dependent variable $\{\iota(X^i_t)\}_{i=1}^N$ at several time stamps $t$, {\it i.e.}, 
%we view 
%\begin{equation}
%    \iota(X_t) = \langle l, S^M(\hat{B}_{0:t}) \rangle + \epsilon_{ls}, \,\,\,\, \epsilon_{ls}\sim\mcN(0, \sigma_{ls}^2).
%\end{equation}
%and approximate 
\begin{align}\label{eq:ols}
    \hat l = & \argmin_{\bm\beta} \|\bm y - \bm X \bm\beta\|^2,  \\
    &\bm y = \{\iota(X^i_t)\}_{i=1}^N,\;    \bm X = \{S^M(\hat{B}^i_{0:t})\}_{i=1}^N. \nonumber
\end{align}
In all experiments in Section~\ref{sec:numerics}, we get decent approximations of $\mu$ on $[0,T]$ by considering only three time stamps $t=0, \frac{T}{2}, T$. Note that such a framework can also deal with multi-dimensional $\iota$, where the regression coefficients become a matrix. 

The choice in \eqref{eq:ols} is mainly motivated by Lemma~\ref{lemma:propbysig} stating $l$ is a linear functional, and by the probability model underlying ordinary linear regression (OLS) which interprets that the least square minimization \eqref{eq:ols} gives the best prediction of $E[\bm y | \bm X]$ restricting to linear relations. There are other benefits for choosing OLS: Once $\hat l$ is obtained in \eqref{eq:ols}, the prediction for unseen common paths is efficient: $    \mu_t(\tilde\omega)\approx \langle \hat l, S^M(\hat{B}_{0:t}(\tilde\omega)) \rangle \,\, \text{for any}\; \tilde\omega\; \text{and}\; t.$
%\begin{equation}
%    \mu_t(\tilde\omega)\approx \langle \hat l, S^M(\hat{B}_{0:t}(\tilde\omega)) \rangle \,\,\,\, \text{for any}\; \tilde\omega\; \text{and}\; t.
%\end{equation}
Moreover, it is easy to integrate with fictitious play: averaging $\mu_t^{(n)}$ from different iterations, commonly needed in fictitious play, now means simply averaging  $\hat l^{(n)}$ over $n$. Next, we analyze the temporal and spatial complexity of using signatures and linear regression as below. 

%thirdly, the easiness to integrated with fictitious play; and lastly it is universal for all common noise paths, easy for generalization.
% provides the reasoning behind using regressions, and we explain it here. 

%Moreover, it has the following advantage. 
%This linear functional $l$ is identical for all $\{B^i\}_{i=1}^N$ and once we have found it, $\{\mu^i\}$ can be calculated by 

%Indeed, by the uniqueness of signature and factorial decay property, we can write $\mu_t = \EE[\iota(X_t)|B_{0:t}] = \EE[\iota(X_t)|S(B_{0:t})] \approx \EE[\iota(X_t)|S^M(\hat{B}_{0,t})]$ 
%and this is the Ordinary Lease Square (OLS) estimation of the linear regression problem

%we shall have the corresponding number of $l$'s with each $l$ approximating one dimension of the distribution. \rh{$l$ can be a matrix, when $\iota$ is multi-dimensional}
%As in Experiment \ref{sec:InvestConsump}, if $\mu_t=\tau(\EE[\iota(X_t)|\F^B_t])$ with $\tau:\RR\to\RR$ continuous, we estimate $\EE[\iota(X_t)|\F^B_t]$ first and then $\mu_t\approx \tau(\langle l, S^M(\hat{B}_{0:t}) \rangle)$. 

%We claim that by using signatures and linear regression, propagating distributions can be done with both time and space complexity linear in sample size $N$. 

%{\bf Time Complexity: {$O(N^2L)$ v.s. $O(NLp+Np^2)$}}
{\it Temporal Complexity:} Suppose we discretize $[0,T]$ into $L$ time stamps: $0 = t_0 \leq t_1 \leq \ldots \leq t_L = T$, and simulate $N$ paths of $W, B$ and $X_t$. The simulation cost is of $\mc{O}(NL)$. For computing the truncated signature $S^M(\hat B)$ of depth $M$, we use the Python package Signatory \cite{kidger2020signatory}, yielding a complexity of $\mc{O}(NLp)$ where $p=\frac{(n_0+1)^{M+1}-1}{n_0}=\mc{O}(n_0^{M})$. Note that one can choose a large $N$ and reuse all sampled common noise paths $B$ for each iteration of fictitious play, thus the computation of $S^M(B)$ is done only once, and $S^M(\hat{B}_{0:t})$ is accessible in constant time for all $t$. The linear regression\footnote{We use the Python package scikit-learn \cite{scikit-learn} to do the linear regression.} (or Ridge regression) takes time $\mc{O}(Np^2)$. Thus, the total temporal complexity is of $\mc{O}(NLp+Np^2)$, which is linear in $N$ given\footnote{$M$ is usually small due to the factorial decay property of the signature. For $n_0$ not large, we have $p\ll N$.} $p\ll N$. Comparing to the nested-loop algorithm, where the cost of simulating SDEs is $\mc{O}(N^2L)$ and computing conditional distribution flows takes time $\mc{O}(N^2L)$, we claim that our algorithm reduced the temporal complexity by a factor of the sample size $N$ by using signatures. 

%To focus on propagating distributions, we ignore the optimization cost at this moment. By using signatures, Thus, simulating  SDEs is of  $\mc{O}(NL)$.
%Signatures of paths will be computed by a Python package Signatory \cite{kidger2020signatory} efficiently with . Denote $p=\frac{(n_0+1)^{M+1}-1}{n_0}=\mc{O}(n_0^{M})$. 

%With pre-computations of signatures over $\{\hat{B}^k\}_{k=1}^N$, which only needs to be done once at the beginning with $\mc{O}(NLp)$ time, $S(\hat{B}_{0:t})$ can be computed in constant time for all $t$ and paths during the propagation of distributions. Hence it takes $\mc{O}(NLp)$ to compute the conditional distribution flows where $p$ stands for the cost of calculating inner products.

%So we end with $O(N^2L)$ time complexity in total.

{\it Spatial Complexity:} In fictitious play, one may choose to average all past flow of measures $\mu^{(n)}$ as the given measures in \eqref{def:J}--\eqref{def:Xt} for the current iteration. Using signatures simplifies it to average $\hat l^{(n)}$. To update it between iterations, one needs to store the current average which costs $\mc{O}(p)$ of the memory. Combining $\mc{O}(NL)$ and $\mc{O}(NLp)$ for storing SDEs and truncated signatures, the overall spacial complexity is  $\mc{O}(NLp)$. The complexity of the nested-loop case is again $\mc{O}(N^2L)$, which we reduce by a factor of $N$. 
%space to store all SDEs, Sig-DFP only needs $O(NL(p+1))$ memory to store signatures of common noises and SDEs. 

%needs to average over all history distribution flows. It is necessary to store and average over all distribution flows for each common noise $B$ at every time stamp during each round of fictitious play. This can be achieved easily by only storing and taking average over the linear functional $l$, which is of order $O(p)$ in time and memory by using signature.

We conclude this section by the following remark: For the general case $\mu_t = \mc{L}(X_t\vert \mcF_t^B)$, though the linear regression is no longer available, the one-to-one mapping between $\mu$ and  $S(\hat{B})$ persists. Therefore, one can train a Generative Adversarial Network (GAN, \citet{goodfellow2014generative}) for generating samples following the distribution $\mu$ by taking truncated signatures as part of the network inputs.
%In general, $\mu$ does not necessarily have conditional expectation form of $\EE[\iota(X_t)\vert \mcF_t^B]$, which disable us to do linear regression. However, there still exits a one-to-one mapping between $\mu$ and $S(\hat{B})$. We can consider 
%\end{rem}

\subsection{Deep Learning Algorithm}\label{sec:Sig-DFP}
\begin{figure}
    \centering
    \includegraphics[width = 0.45\textwidth]{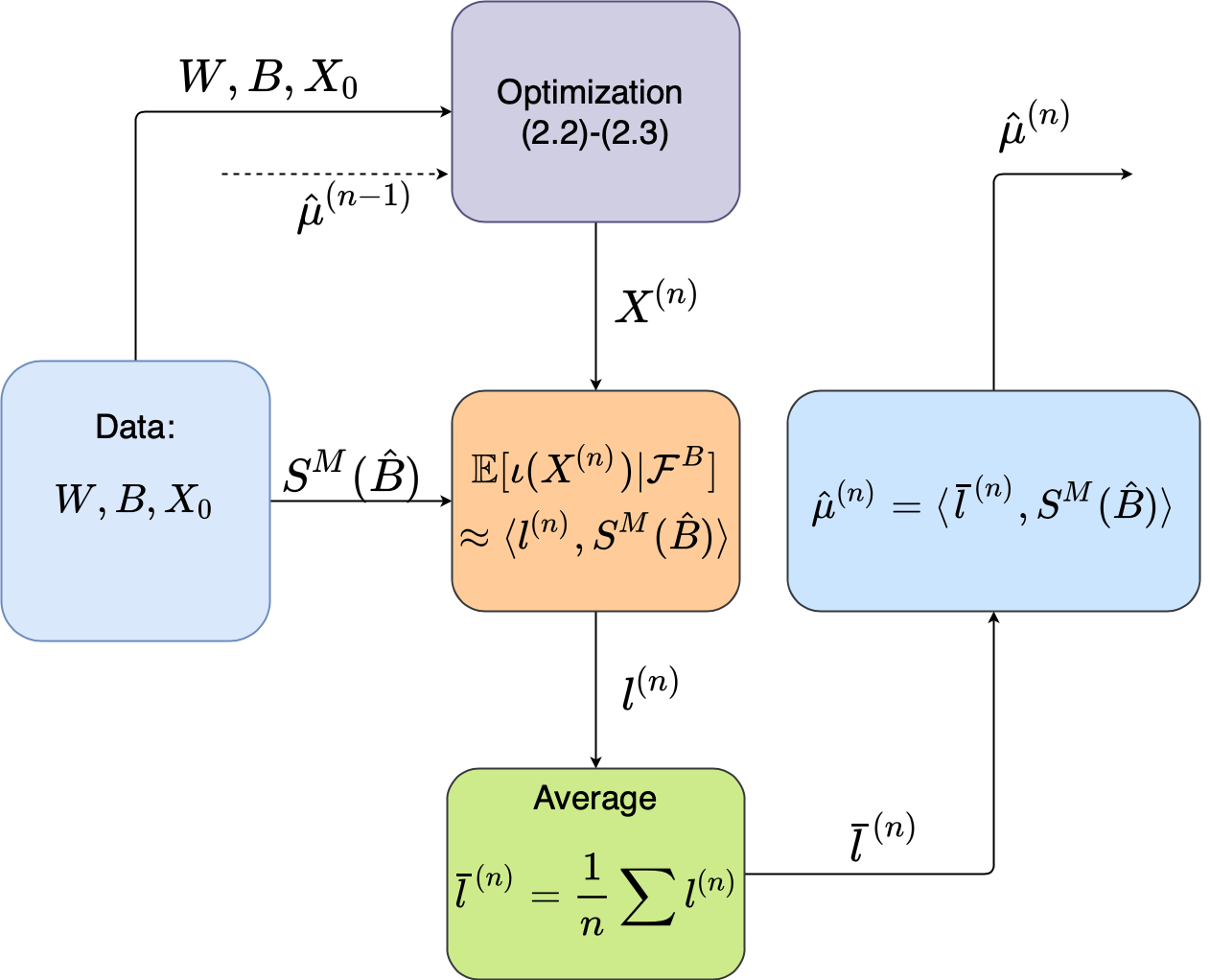}
    \caption{Flowchart of one iteration in the Sig-DFP Algorithm. Input: idiosyncratic noise $W$, common noise $B$, initial position $X_0$ and measure flow $\hat{\mu}^{(n-1)}$ from the last iteration. Output: measure flow $\hat{\mu}^{(n)}$ for the next iteration.}
    \label{fig:algo}
\end{figure}

Having explained the key idea on how to approximate $\mu$ efficiently, we describe the Sig-DFP algorithm in this subsection. The algorithm consists of repeatedly solving \eqref{def:J}--\eqref{def:Xt} for a given measure flow $\mu$ using deep learning in the spirit of \citet{han2016deep}, and passing the yielded $\mu$ to the next iteration by using signatures. The flowchart of the idea is illustrated in Figure \ref{fig:algo}. Consider a partition $\pi$ of $[0,T]: 0=t_0<\cdots<t_L=T$, denote by $\hat \mu^{(n-1)}$ the given flow of measures at stage $n$, the stochastic optimal control problem \eqref{def:J}--\eqref{def:Xt} is solved by
%We use the Euler scheme to discretize the SDEs \eqref{def:Xt}, and then the optimal control problem \eqref{def:J}--\eqref{def:Xt} becomes, for a partition $\pi$ of $[0,T]$, $0=t_0<\cdots<t_L=T$, and $X^i_0 \sim \mu_0$,
\begin{align}
    & \inf_{\{\alpha_k\}_{k=0}^{N-1}} \frac{1}{N}\sum_{i=1}^N \biggl(\sum_{k=0}^{L-1}f(t_k, X_k^i, \hat\mu^{(n-1)}_k(\omega^i), \alpha^i_k) \Delta_k \nonumber \\
    & \hspace{6em}+ g(X^i_L, \hat\mu_L^{(n-1)}(\omega^i)) \biggr), \label{def:J_discrete}
    \\
    & \text{where } X^{i}_{k+1} = X^{i}_{k} + b(t_k, X^i_k, \hat\mu^{(n-1)}_k(\omega^i), \alpha^i_k) \Delta_k \nonumber \\ & \hspace{6em} + \sigma(t_k, X^i_k, \hat\mu^{(n-1)}_k(\omega^i), \alpha^i_k) \Delta W^i_k \nonumber \\
      & \hspace{6em}+ \sigma^0(t_k, X^i_k, \hat\mu^{(n-1)}_k(\omega^i), \alpha^i_t) \Delta B^i_k,  \label{def:Xt_discrete} 
\end{align}
where we replace the subscript $t_k$ by $k$ to simplify notations, and let $\Delta_k = t_{k+1}-t_k$, $\Delta W^i_k = W^i_{t_{k+1}} - W^i_{t_{k}}$, $\Delta B^i_k = B^i_{t_{k+1}} - B^i_{t_{k}}$. Here, we use the superscript $i$ to represent the $i^{th}$ sample path and $\hat\mu^{(n-1)}_k(\omega^i)$ to emphasize the stochastic measure's dependence on the $i^{th}$ sample path of $B$ up to time $t_k$. The control $\alpha_k$ is then parameterized by neural networks (NNs) in the feedback form:
\begin{equation}
     \alpha_k^i:= \alpha_\varphi(t_k, X^i_k, \hat\mu^{(n-1)}_k(\omega^i); \varphi),
\end{equation}
where $\alpha_\varphi$ denotes the NN map with parameters $\varphi$, and searching the infimum in \eqref{def:J_discrete} is translated into minimizing $\varphi$. The yielded optimizer $\varphi^\ast$ gives $\alpha_k^{i, \ast}$, with which the optimized state process paths $\{X^{i, \ast}\}_{i=1}^N$ are simulated and its conditional law $\mc{L}(X^{\ast} \vert \mc{F}^B)$, denoted by $\mu^{(n)}$, is approximated using signatures as described in Section~\ref{sec:propsig}. This finishes one iteration of fictitious play. Denote by $\tilde\mu^{(n)}$ the approximation of $\mu^{(n)}$, we then pass $\tilde\mu^{(n)}$ to the next iteration via updating $\hat\mu^{(n)} = \frac{1}{n}\tilde\mu^{(n)} + \frac{n-1}{n}\hat \mu^{(n-1)}$ by averaging the coefficients in \eqref{eq:ols}.

We summarize it in Algorithm \ref{alg:sig-dfp}, with implementation details deferred to Appendix~\ref{app:SigDFP}. Note that the simulation of $X^{i, (n)}$ and $J_B(\varphi, \Bar{\mu}^{(n-1)})$ uses the equations \eqref{def:Xt_algo} and \eqref{def:J_algo} in Appendix~\ref{app:SigDFP}, respectively. 

%and hence complete steps (2)\&(3) for fictitious play. 

%{\it Implementation} is described in details in Algorithm \ref{alg:sig-dfp}.

\begin{algorithm}[tb]
   \caption{The Sig-DFP Algorithm}
   \label{alg:sig-dfp}
\begin{algorithmic}
   \STATE {\bfseries Input:} $b, \sigma, \sigma_0, f, g, \iota$ and $X^i_0, (W^i_{t_k})_{k=0}^L, (B^i_{t_k})_{k=0}^L$ for $i=1,2,\dots, N$; $N_{\text{round}}$: rounds for FP; 
   
   $B$: minibatch size; $N_{\text{batch}}$: number of minibatches. 
   \STATE Compute the signatures of $\hat{B}^i_{0:t_k}$ for $i=1, \dots, N$, $k=1,\dots, L$;
   \STATE Initialize $\hat{\mu}^{(0)}$, $\varphi$;
   %\REPEAT
   \FOR{$n=1$ {\bfseries to} $N_{\text{round}}$}
   \FOR{$r=1$ {\bfseries to} $N_{\text{batch}}$}
   \STATE Simulate the $r^{th}$ minibatch of $X^{i,(n)}$ using $\hat{\mu}^{(n-1)}$ and compute $J_B(\varphi, \hat{\mu}^{(n-1)})$;
   \STATE Minimize $J_B(\varphi, \hat{\mu}^{(n-1)})$ over $\varphi$, then update $\alpha_\varphi$;
   \ENDFOR
   \STATE Simulate $X^{i, (n)}$ with the optimized $\alpha_\varphi^\ast$, 
   %and $\iota(X^{i, (n)}_0), \iota(X^{i, (n)}_{L/2}), \iota(X^{i, (n)}_L)$, 
   for $i=1, \dots, N$;
   \STATE Regress $\iota(X^{i, (n)}_0), \iota(X^{i, (n)}_{L/2}), \iota(X^{i, (n)}_L)$ on $S^M(\hat{B}^i_{0:0})$, $S^M(\hat{B}^i_{0:t_{L/2}})$, $S^M(\hat{B}^i_{0:t_L})$ to get $l^{(n)}$; 
   %, $i=1,\dots, N$,
   \STATE Update $\bar{l}^{(n)} = \frac{n-1}{n}\bar{l}^{(n-1)} + \frac{1}{n}l^{(n)}$;
   \STATE Compute $\hat{\mu}^{(n)}$ by $\hat{\mu}^{(n)}_k(\omega^i) = \langle \Bar{l}^{(n)}, S^M(\hat{B}^i_{0:t_k}) \rangle$, for $i=1,2,\dots, N, k=1,\dots, L$;
   \ENDFOR
   %\UNTIL{$noChange$ is $true$}
   \STATE {\bfseries Output:} the optimized $\alpha_\varphi^\ast$ and  $\bar{l}^{(N_\text{round})}$.
\end{algorithmic}
\end{algorithm}

%{\it Convergence analysis.} 

\begin{thm}[Convergence analysis]\label{thm:cvg} 
 Let $(\alpha^\ast$,$\mu^\ast)$ be the mean-field equilibrium in Definition~\ref{defn:MFG},  $\alpha^{(n)}$ be the optimal control, and $\mu^{(n)}$ be the measure flow of the optimized state process after the $n^{th}$ iteration of fictitious play, and $\tilde\mu^{(n)}$ be the approximation by truncated signatures. Under Assumption~\ref{assump:cvg} and $\displaystyle \sup_{t \in [0,T]} \EE[\mc{W}_2^2(\tilde{\mu}^{(n)}_t, \mu^{(n)}_t)] \leq \eps$, we have
\begin{equation}
\begin{aligned}
     \sup_{t \in [0,T]} \EE&[\mc{W}_2^2(\tilde\mu^{(n)}_t, \mu^\ast_t)] + \int_0^T \EE|\alpha_t^{(n)} - \alpha_t^\ast|^2 \ud t  \notag \\
   &\leq C(q^n \sup_{t \in [0,T]} \EE[\mc{W}_2^2(\mu^{(0)}_t, \mu^\ast_t)] + \eps),\label{eq:cvg}
   \end{aligned}
\end{equation}
  for some constants $C >0$ and $0 < q < 1$, where $\mc{W}_2$ denotes the 2-Wasserstein metric. 
\end{thm}

Moreover, if we consider a partition of $[0,T]: 0=t_0<\cdots<t_L=T$, and define $\pi(t) = t_k$ for $t \in [t_k, t_{k+1})$ with $\|\pi\| = \max_{1 \leq k < L}|t_{k} - t_{k-1}|$, then

\begin{thm}[Convergence in discrete time]\label{thm:numcvg} Let $\mu_{t_k}^{(n)}$ be the conditional law of the discretized optimal process $X_{t_k}^{(n)}$ after the $n^{th}$ iteration of fictitious play (cf. \eqref{def:Xt_discrete}), and $\tilde\mu^{(n)}_{t_k}$ be the approximation by truncated signatures. Under Assumption~\ref{assump:cvg} and $\displaystyle \sup_{0 \leq k \leq L} \EE[\mc{W}_2^2(\tilde{\mu}^{(n)}_{t_k}, \mu^{(n)}_{t_k})] \leq \eps$, one has
\begin{equation}
\begin{aligned}
     \sup_{t \in [0,T]} \EE&[\mc{W}_2^2(\tilde\mu^{(n)}_{\pi(t)}, \mu^\ast_t)] + \int_0^T \EE|\alpha_{\pi(t)}^{(n)} - \alpha_t^\ast|^2 \ud t  \notag \\
   &\leq C(q^n  \sup_{0 \leq k \leq L}\EE[\mc{W}_2^2(\mu^{(0)}_{t_k}, \mu^\ast_{t_k})] + \eps + \|\pi\|),
   \end{aligned}
\end{equation} 
for some constants $C >0$ and $0 < q < 1$, where $\alpha_{t_k}^{(n)} = \hat \alpha(t_k, X_{t_k}, Y_{t_k}, \tilde \mu_{t_k}^{(n-1)})$, and $(X_t, Y_t)$ solves \eqref{def:FBSDE} with $\mu$ replaced by $\tilde \mu_{t_k}^{(n-1)}$.
\end{thm}

The proofs of Theorems~\ref{thm:cvg} and \ref{thm:numcvg} are given in Appendix~\ref{app:algorithm} due to the page limit. 
%See Appendix~\ref{app:algorithm} for proofs due to the page limit.

%\revise{
%Moreover, if we consider a partition of $[0,T]: 0=t_0<\cdots<t_L=T$, and define $\pi(t) = t_k$ for $t \in [t_k, t_{k+1})$ with $\|\pi\| = \max_{1 \leq k < L}|t_{k} - t_{k-1}|$, then we can have the convergence in discrete time with a term proportional to $\|\pi\|$ adding to the right-hand-side of \eqref{eq:cvg}. %We shall provide more details about the numerical convergence 
%}

%\begin{equation}
%    \sup_{t \in [0,T]} \EE[\mc{W}_2(\tilde{\mu}^{(n)}_t, \mu^{(n)})] \leq \eps,
%\end{equation}
%we have,

Remark that the Sig-DFP framework is flexible. We choose to solve \eqref{def:J}-\eqref{def:Xt} by direct parameterizing control policies $\alpha_t$ for the sake of easy implementation and the possible exploration of multiple mean-field equilibria. If the equilibrium is unique, with proper conditions on the coefficients $b, \sigma, \sigma^0, f$ and $g$, one can reformulate \eqref{def:J}-\eqref{def:Xt} into McKean-Vlasov FBSDEs or stochastic FP/HJB equations, and solve them by fictitious play and propagating the common noise using signatures.

%One final remark regarding using signatures to propagate measure flows. Remark that propagation of distribution by signatures also has the following advantages: 1. The optimization costs are increasing under nested loop case for both FBSDE and stochastic Fokker-Planck/Hamilton-Jacobi-Bellman approaches because of the appearance of common noise. Propagation distributions by signature can be immediately used under these two frameworks with non-increasing optimization complexity; 2. In the case that one wants to feed newly generated common noises in every round of fictitious play, which is not applicable through nested loop, it can be simply finished by passing the linear functional $l^{(n)}$ into next round.
%\rh{propagation of distribution can also be used in two other approaches, to avoided nested loop. }

\section{Experiments}\label{sec:numerics}

In this section, we present the performance of Sig-DFP for three examples: homogeneous, heterogeneous, and heterogeneous extended MFGs. A relative $L^2$ metric will be used for performance measurement, defined for progressively measurable random processes as
\begin{equation}
    L^2_R(x, \hat{x}) := \sqrt{\frac{\EE[\int_0^T \|x_t-\hat{x}_t\|^2\ud t]}{\EE[\int_0^T \|x_t\|^2\ud t]}}, \label{def: rela_L2}
\end{equation}
where $x$ is a benchmark process and $\hat{x}$ is its prediction. We shall use stochastic gradient descent (SGD) optimizer for all three experiments. Training processes are done on a server with Intel Core i9-9820X (10 cores, 3.30 GHz) and RTX 2080 Ti GPU, and training time will be reported in Appendix \ref{app:SigDFP}. Implementation codes are available at \url{https://github.com/mmin0/SigDFP}.

%Details of benchmark solutions for the three examples are provided in Appendix~\ref{app:Benchmark}.

{\bf Data Preparation.} For all three experiments, the size of both training and test data is  $N=2^{15}$, and the size of validation data is $N/2$. We fix $T=1$ and discretize $[0,1]$ by $t_{k}=\frac{k}{100},\; k=0,1,\dots,100$. Initial states are generated independently by $X_0^i \sim\mu_0$, with $\mu_0= U(0, 1)$ as the uniform distribution. The idiosyncratic Brownian motions $W$ and common noises $B$ are generated by antithetic variates for variance reduction, {\it i.e.}, we generate the first half samples $(W^i, B^i)$ and get the other half $(-W^i, -B^i)$ by flipping.

{\bf Benchmarks.} The examples below are carefully chosen with analytical benchmark solutions. Due to the space limit, we provide the details in Appendix~\ref{app:Benchmark}.

%\subsection{Linear-Quadratic MFG}
{\bf Linear-Quadratic MFGs.} We first consider a Linear-Quadratic MFG with common noise proposed in \citet{carmona2015mean}, formulated as below:
\begin{align} %(\alpha_t)_{0 \leq t \leq T}
    & \inf_{\alpha} \EE\biggl\{\int_0^T \left[
    \frac{\alpha_t^2}{2}-q\alpha_t(m_t-X_t)+\frac{\epsilon}{2}(m_t-X_t)^2
    \right]\ud t \nonumber \\ & \hspace{5em}+\frac{c}{2}(m_T-X_T)^2 \biggr\}, \\
    &\text{where }
    \ud X_t = [a(m_t-X_t)+\alpha_t]\ud t \nonumber\\& \hspace{7em} + \sigma(\rho \ud B_t + \sqrt{1-\rho^2}\ud W_t). \label{def:LQ_SDE}
\end{align}
Here $m_t = \EE[X_t|\mcF^B_t]$ is the conditional population mean, $\rho \in[0,1]$ characterizes the noise correlation between agents, and $q, \epsilon, c, a, \sigma$ are positive constants. The agents have homogeneous preferences and aim to minimize their individual costs. We assume $q\le \epsilon^2$ so that the Hamiltonian is jointly convex in state and control variables, ensuring a unique mean-field equilibrium.

%\begin{align}
%    & \inf_{(\alpha_t)_{0\le t\le T}} \EE\left\{\int_0^T \left[
%    \frac{\alpha_t^2}{2}-q\alpha_t(m_t-X_t)+\frac{\epsilon}{2}(m_t-X_t)^2
%    \right]\ud t +\frac{c}{2}(m_T-X_T)^2 \right\}, \\
%    &\text{subject to: }
%    \ud X_t = [a(m_t-X_t)+\alpha_t]\ud t + \sigma(\rho \ud B_t + \sqrt{1-\rho^2}\ud W_t), \label{def:LQ_SDE}
%\end{align}

{\it Training \& Results.} $\alpha_\varphi$ is a feedforward NN with two hidden layers of width 64. The truncated signature depth is chosen at $M=2$. The model is trained for $500$ iterations of fictitious play. The optimized state process $\hat{X}$ and its conditional mean $\hat{m}$ generated by test data are shown in Figures~\ref{fig:LQ}a and \ref{fig:LQ}b. The minimized cost after each iteration computed using validation data is given in Figure~\ref{fig:LQ}c, where one can see a rapid convergence to the benchmark cost. During the experiments, we notice a slow convergence speed when using the average of $m^{(n)}$ in \eqref{def:LQ_SDE}. This is because the initial guess $m^{(0)}$ is in general far from the truth. Therefore, for the first half of iterations, we simply use the previous-step result $m^{(n-1)}$. The learning rate is set as 0.1 for the first half and 0.01 for the second half of training. The relative $L^2$ errors for test data are listed in Table \ref{tab:LQ}.
%we only apply averaging over distributions (or linear functions) after the 250th iteration of fictitious play. We set learning rate as $0.1$ for the first half iterations and $0.01$ for the second half. The relative $L^2$ distances for test data are listed in Table \ref{tab:LQ}.

\begin{table}[ht]
\vspace{-1em}
\caption{Relative $L^2$ errors on test data for the LQ MFG.}
\label{tab:LQ}
\vskip 0.1in
%\vspace{0.5em}
    \begin{center}
    \begin{small}
    \begin{sc}
    \begin{tabular}{lcccr}
    \toprule
         &  SDE $X_t$ & Control $\alpha_t$ & Equilibrium $m_t$\\
         \midrule
        $L^2_R$ & $0.0031$  & $0.0044$ & $0.058$ \\
        \bottomrule
    \end{tabular}
    \end{sc}
    \end{small}
    \end{center}
    \vspace{-1em}
\end{table}

\begin{figure*}[ht]
    \centering
    \subfloat[$X_t$]{
         \includegraphics[width=0.6\columnwidth]{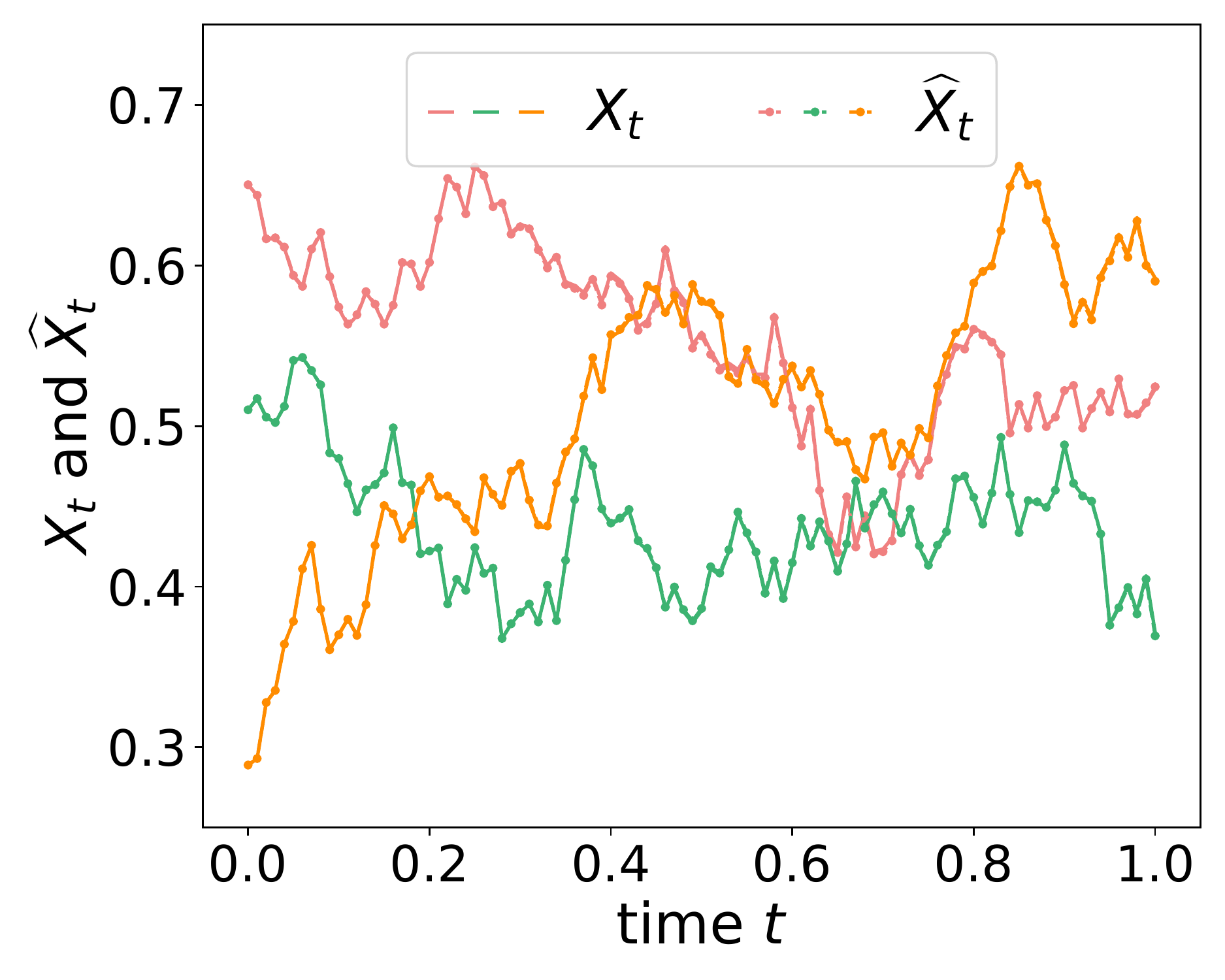}}
         %\caption{$X_t$}
%    \subfloat[$\alpha_t$]{
%         \includegraphics[width=0.45\columnwidth]{SystemicRisk/alpha.pdf}}
    \subfloat[$m_t = \EE(X_t \vert \mcF_t^B)$]{
         \includegraphics[width=0.6\columnwidth]{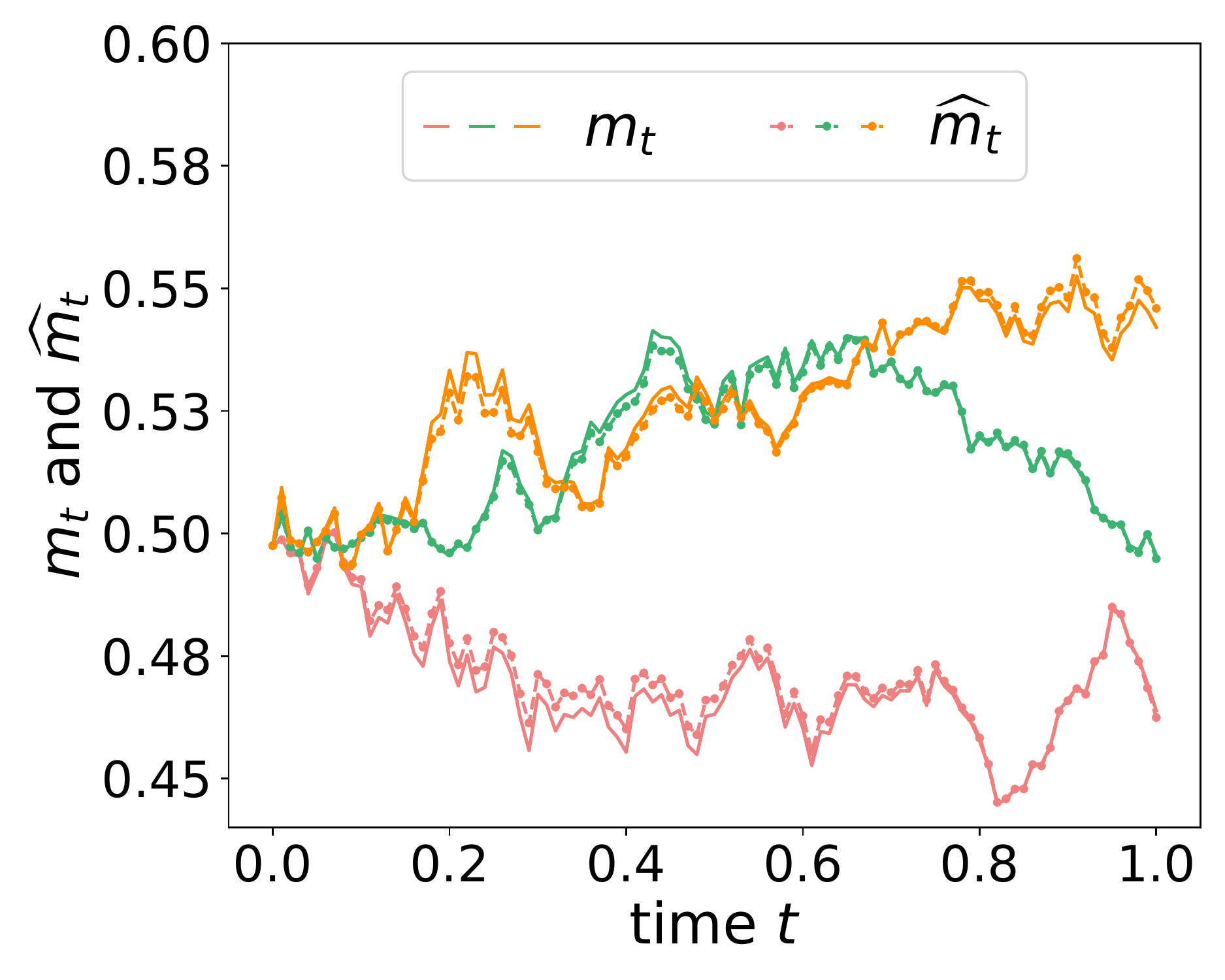}}
    \subfloat[Minimized Cost]{
         \includegraphics[width=0.6\columnwidth]{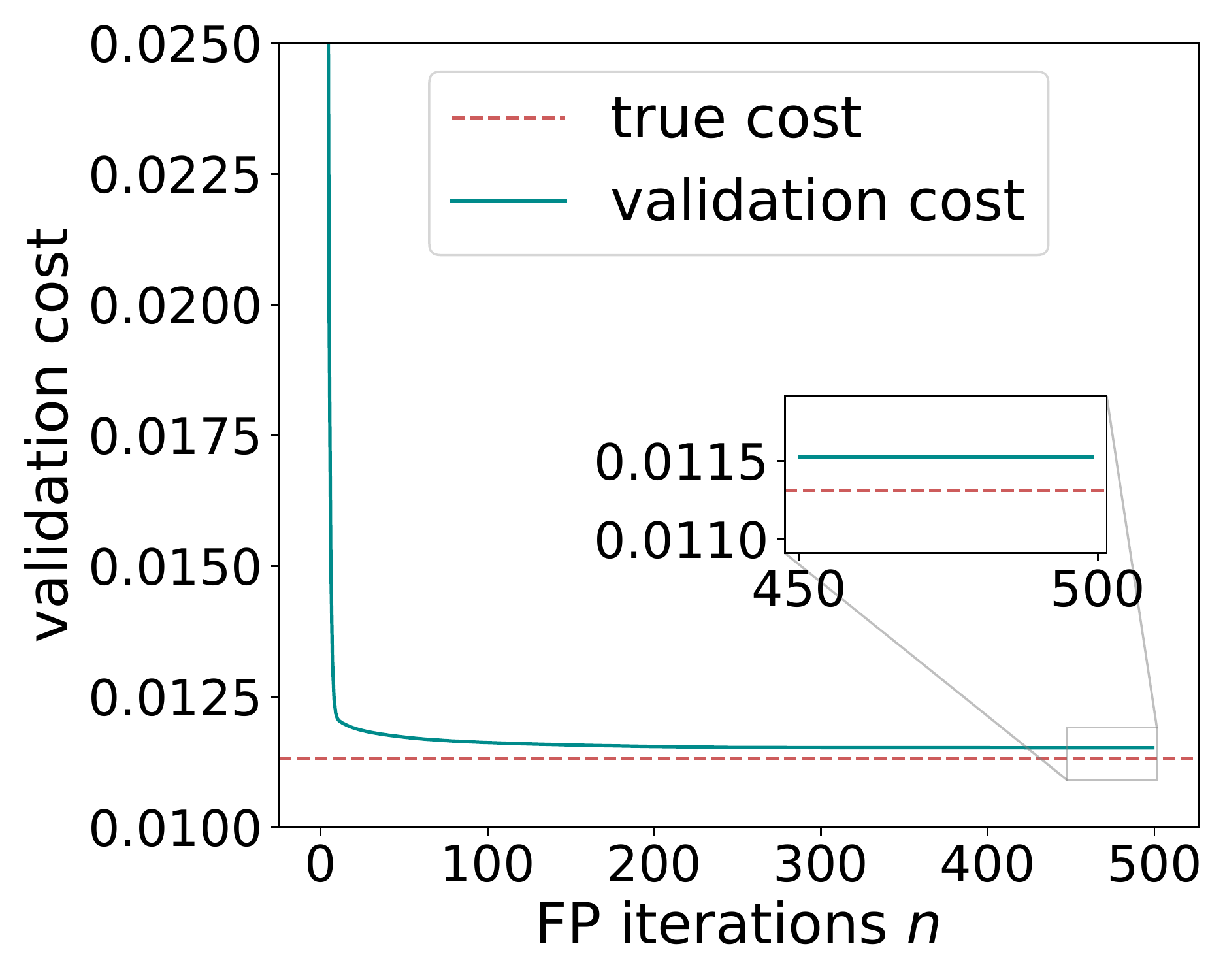}}
    \caption{Panels (a) and (b) give three trajectories of $X_t$, $m_t = \EE[X_t \vert \mcF_b^B]$ (solid lines) and their approximations (dashed lines) using different $(X_0, W, B)$ from test data. Panel (c) shows the minimized cost computed using validation data over fictitious play iterations.
    %(a) Controlled SDE $X_t$; (b) Conditional distribution flow $m_t$; (c) Costs of validation data over fictitious play iterations. 
    Parameter choices are: $\sigma=0.2, q=1, a=1, \epsilon=1.5, \rho=0.2, c=1$, $x_0 \sim U(0,1)$.}
    \label{fig:LQ}
\end{figure*}

{\bf Mean-Field Portfolio Game.} 
Our second experiment is performed on a heterogeneous MFG proposed by \citet{lacker2018mean}, where the agent's preference is different, characterized by a type vector $\zeta$ which is random and drawn at time 0. They all aim to maximize their exponential utility of terminal wealth compared to the population average:
%optimize their portfolio the their utility is based on relative performance,
% considered an optimal portfolio management problem under competition and relative performance criteria under the mean field regime. More specifically, there are different agents who are characterized by their type vector (a random vector) $\zeta=(\xi, \delta, \theta, \mu, \nu, \sigma)$. All agents use exponential (CARA) utilities and their goal is to maximized expected payoff given by
\begin{equation}
    \sup_{\pi} \EE\left[ -\exp{\left(-\frac{1}{\delta}(X_T-\theta m_T) \right)} \right],
\end{equation}
where the dynamics are 
\begin{equation}
    \ud X_t = \pi_t(\mu \ud t + \nu \ud W_t + \sigma \ud B_t), \quad X_0=\xi.
\end{equation}

Here $m$ represents the conditional mean $m_t:=\EE[X_t|\mcF^B_t]$, and $\zeta=(\xi, \delta, \theta, \mu, \nu, \sigma)$ is random.

{\it Training \& Results.} We use truncated signatures of depth $M=2$ and a feedforward NN $\pi_\varphi$ with 4 hidden layers\footnote{Since agents are heterogeneous characterized by their type vectors $\zeta$, $\pi_\varphi$ takes $(\zeta, t, X_t, m_t)$ as inputs. Hidden neurons in each layer are (64, 32, 32, 16).} to approximate $\pi$. We train our model with 500 iterations of fictitious play. The learning rate starts at $0.1$ and is reduced by a factor of $5$ every 200 rounds. The relative $L^2$ errors evaluated under test data are listed in Table \ref{tab:Invest}. Figure \ref{fig:Invest} compares $X$ and $m$ to their approximations, and plots the maximized utilities.

%We plot $X$, $m$ and their approximations, and the maximized utilities in Figure \ref{fig:Invest}.

\begin{table}[ht]
\vspace{-1em}
\caption{Relative $L^2$ errors on test data for MF Portfolio Game.}
\label{tab:Invest}
\vskip 0.1in
%\vspace{-0.7em}
    \begin{center}
    \begin{small}
    \begin{sc}
    \begin{tabular}{lcccr}
    \toprule
         &  SDE $X_t$ & Invest $\pi_t$ & Equilibrium $m_t$\\
         \midrule
        $L^2_R$  & $0.068$  & $0.035$ & $0.085$ \\
        \bottomrule
    \end{tabular}
    \end{sc}
    \end{small}
    \end{center}
    \vspace{-1.5em}
\end{table}

%\begin{table}[ht]
%\caption{Relative $L^2$ errors on test data for MF Portfolio Game (PG) and Optimal Consumption and Investment (OCI) MFGs.}
%\label{tab:Invest}
%\vskip 0.1in
%    \begin{center}
%    \begin{small}
%    \begin{sc}
%    \begin{tabular}{lcccr}
%    \toprule
%         &  SDE: $X_t$ & Invest: $\pi_t$ & Equilibrium: $m_t$\\
%         \midrule
%        $L^2_R$ (PG) & $0.068$  & $0.035$ & $0.085$ \\
%        \midrule
%        $L^2_R$ (OCI) & $0.1126$  & $0.0614$ & $0.0279$  & $0.0121$ \\
%        \bottomrule
%    \end{tabular}
%    \end{sc}
%    \end{small}
%    \end{center}
%\end{table}

\begin{figure*}[htb!]
    \centering
    \subfloat[$X_t$]{
         
         \includegraphics[width=0.6\columnwidth]{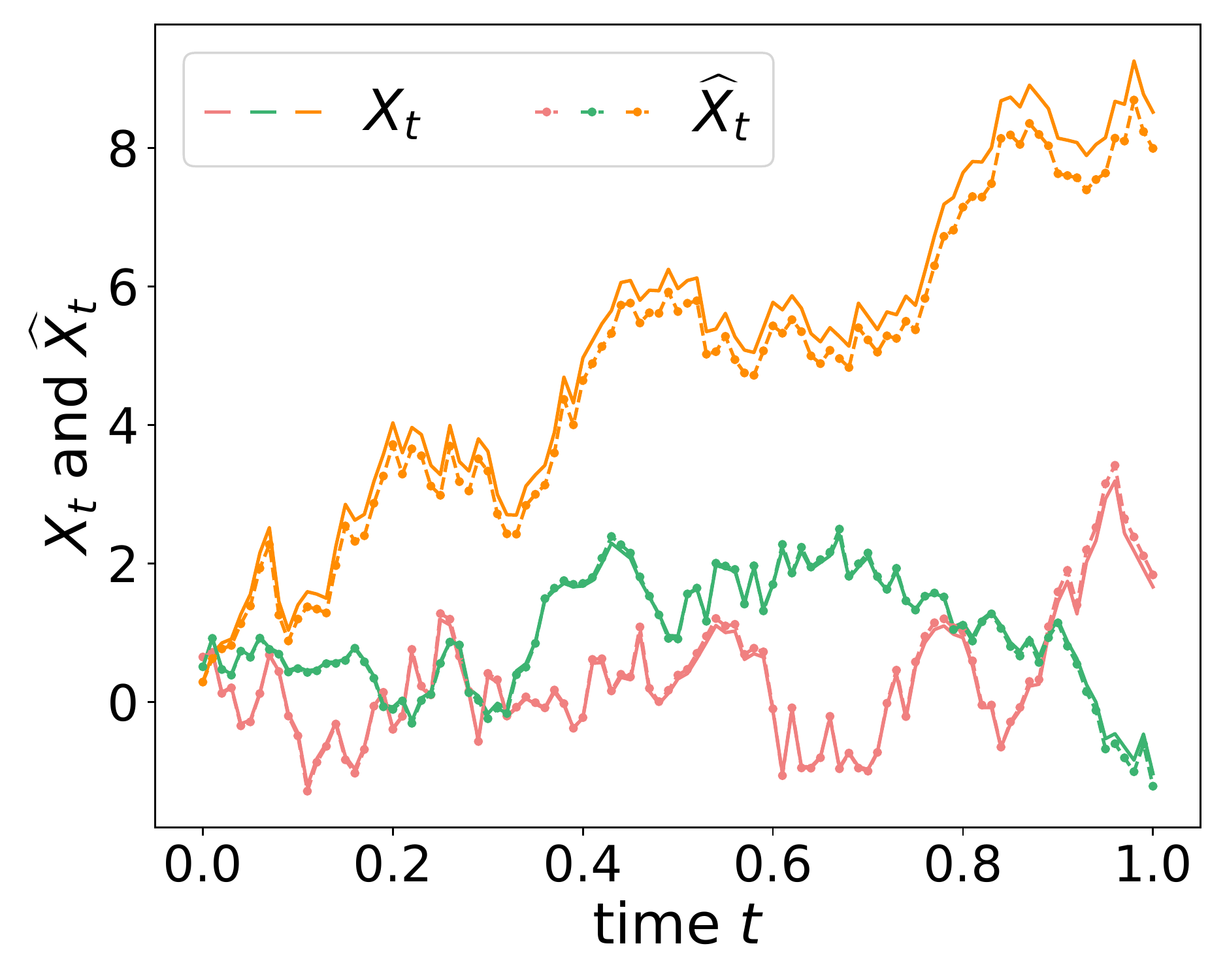}
     }
     %\hfill
     \subfloat[$m_t = \EE(X_t \vert \mcF_t^B)$]{
         
         \includegraphics[width=0.6\columnwidth]{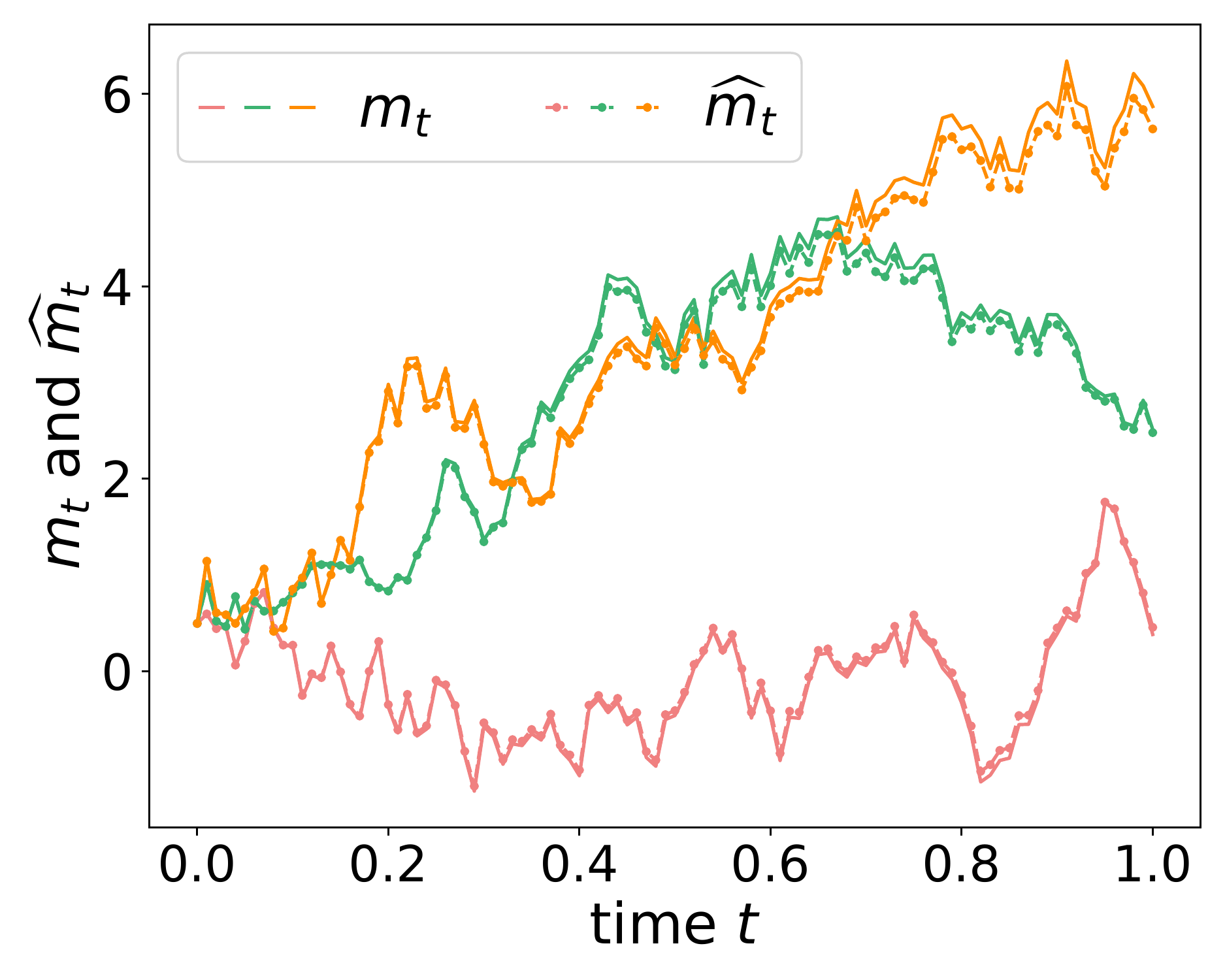}
     }
     %\hfill
     \subfloat[Maximized Utility]{
         
         \includegraphics[width=0.6\columnwidth]{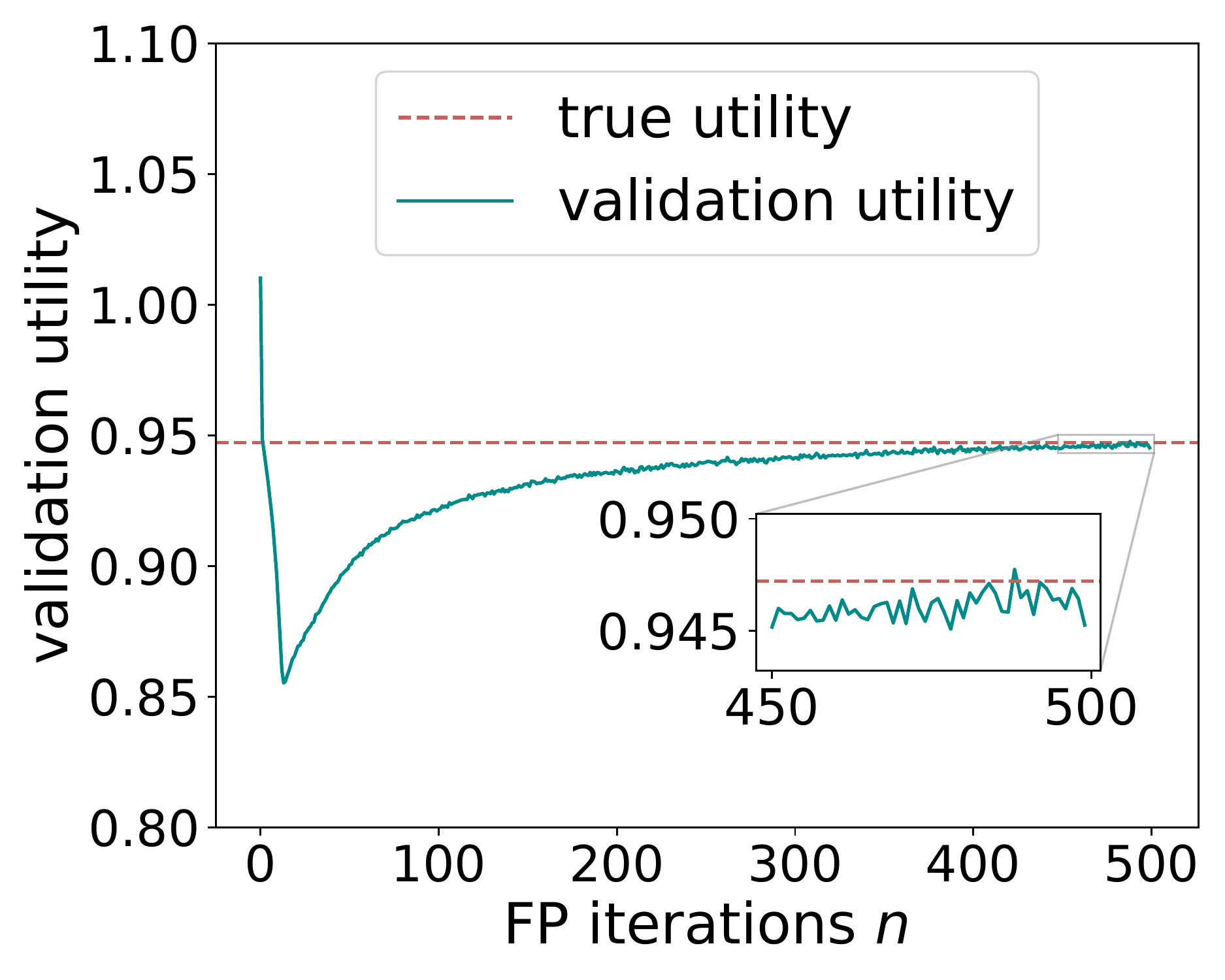}
         
     }
    \caption{Panels (a) and (b) give three trajectories of $X_t$, $m_t= \EE[X_t \vert \mcF_b^B]$ (solid lines) and their approximations (dashed lines) using different $(X_0, W, B)$ from test data. Panel (c) shows the maximized utility computed using validation data over fictitious play iterations.
    %Plots on test data for three different realizations of $(X_0^i, W^i, B^i, \zeta^i)$. Solid line is the benchmark solution and dashed line is the predicted result.  (a) Controlled SDE $X_t$; (b) Conditional distribution flow $m_t:=\EE[X_t|\mcF^B_t]$; (c) Utilities of validation data over fictitious play iterations. 
    Parameter choices are: $\delta \sim U(5, 5.5), \mu\sim U(0.25, 0.35), \nu\sim U(0.2, 0.4), \theta\sim U(0,1), \sigma\sim U(0.2, 0.4)$, $\xi \sim U(0,1)$. }
    \label{fig:Invest}
\end{figure*}

%\subsection{Mean Field Game of Optimal Consumption and Investment}
%\label{sec:InvestConsump}

{\bf Mean-Field Game of Optimal Consumption and Investment.}
Our last experiment considers an extended heterogeneous MFG proposed by \citet{lacker2020many}, where agents interact via both states and controls. The setup is similar to \citet{lacker2018mean} except for including consumption and using power utilities. More precisely, each agent is characterized by a type vector $\zeta=(\xi, \delta, \theta, \mu, \nu, \sigma, \epsilon)$, and the optimization problem reads
%\begin{equation}
 %   \begin{aligned}
  %      \sup_{(\pi_t)_{0\le t\le T}, (c_t)_{0\le t\le T}} \EE\biggl[ 
   % &\int_0^T U(c_t X_t(\Gamma_t m_t)^{-\theta}; \delta)\ud t \\ &+ \epsilon U(X_T m^{-\theta}_T;\delta)
    %\biggr],     \end{aligned}
 %\label{def: InvestConsumpValue}
%\end{equation}
\begin{equation}
        \sup_{\pi, c} \EE\biggl[ 
    \int_0^T U(c_t X_t(\Gamma_t m_t)^{-\theta}; \delta)\ud t + \epsilon U(X_T m^{-\theta}_T;\delta)
    \biggr],     
 \label{def: InvestConsumpValue}
\end{equation}
where 
$U(x; \delta) = \frac{1}{1-\frac{1}{\delta}}x^{1-\frac{1}{\delta}}$, $\delta\ne 1$, 
%$U(x; \delta) = \left\{ \begin{array}{lc}
%    \frac{1}{1-1/\delta}x^{1-1/\delta} & \text{if } \delta\ne 1 \\
%    \log x & \text{if } \delta= 1
%\end{array} \right.$ 
$X_t$ follows
\begin{equation}
    \ud X_t = \pi_t X_t(\mu \ud t + \nu \ud W_t + \sigma \ud B_t) - c_tX_t \ud t,  \label{def: InvestConsumpSDE}
\end{equation}
and $X_0 = \xi$. Here $\Gamma_t = \exp\EE[\log c_t |\mcF^B_t]$ and $m_t = \exp\EE[\log X_t |\mcF^B_t]$ are the mean-field interactions from consumption and wealth. 
%with $\quad X_t, c_t>0 \text{ and } X_0=\xi$.

{\it Training \& Results.} 
For this experiment, we use truncated signatures of depth $M=4$. The optimal controls $(\pi_t, c_t)_{0\le t\le 1}$ are parameterized by two neural networks $\pi_\varphi$ and $c_\varphi$, each with three hidden layers.\footnote{Due to the nature of heterogeneous extended MFG, both $\alpha_\varphi$ and $c_\varphi$ take $(\zeta_t, t, X_t, m_t, \Gamma_t)$ as inputs. Hidden neurons in each layer are (64, 64, 64).} Due to the extended mean-field interaction term $\Gamma_t$, we will propagate two conditional distribution flows, {\it i.e.}, two linear functionals $\Bar{l}^{(n)}, \Bar{l}_c^{(n)}$ during each iteration of fictitious play. Instead of estimating $m_t, \Gamma_t$ directly, we estimate $\EE[\log X_t|\mcF^B_t], \EE[\log c_t|\mcF^B_t]$ by $\langle \Bar{l}^{(n)}, S^4(B_{0:t}) \rangle$, $\langle \Bar{l}_c^{(n)}, S^4(B_{0:t}) \rangle$ and then take exponential to get $m_t, \Gamma_t$. To ensure the non-negativity condition of $X_t$, we evolve $\log X_t$ according to \eqref{def: InvestConsumplogSDE} and then take exponential to get $X_t$. For optimal consumption, $c_\varphi$ is used to predicted $\log c_t$ and thus $\exp c_\varphi$ gives the predicted $c_t$. With 600 iterations of fictitious play and a learning rate of 0.1 decaying by a factor of 5 for every 200 iterations, the relative $L^2$ errors for test data are listed in Table \ref{tab:OCI}. Figure \ref{fig:OCI} compares $X$ and $m$ to their approximations, and plots the maximized utilities. Plots of $\pi_t$, $c_t$, $\Gamma_t=\exp\EE(\log c_t |\mcF^B_t)$ are provided in Appendix~\ref{app:OCI}.

%We show three sampled trajectories in Figure \ref{fig:OCI}.

\begin{table}[htb!]
\vspace{-1.2em}
\caption{Relative $L^2$ errors on test data for Optimal Consumption and Investment MFG.}
\label{tab:OCI}
\vskip 0.1in
%\vspace{-1em}
    \begin{center}
    \begin{small}
    \begin{sc}
    \begin{tabular}{lccccr}
    \toprule
         &  Invest $\pi_t$ & Consumption $c_t$ & $m_t$ & $\Gamma_t$\\
         \midrule
        $L^2_R$ & $0.1126$  & $0.0614$ & $0.0279$  & $0.0121$ \\
        \bottomrule
    \end{tabular}
    \end{sc}
    \end{small}
    \end{center}
    \vspace{-1.5em}
\end{table}

%\begin{onecolumn}

\begin{figure*}[hbt!]
    \centering
    \subfloat[$X_t$]{
         \includegraphics[width=0.6\columnwidth]{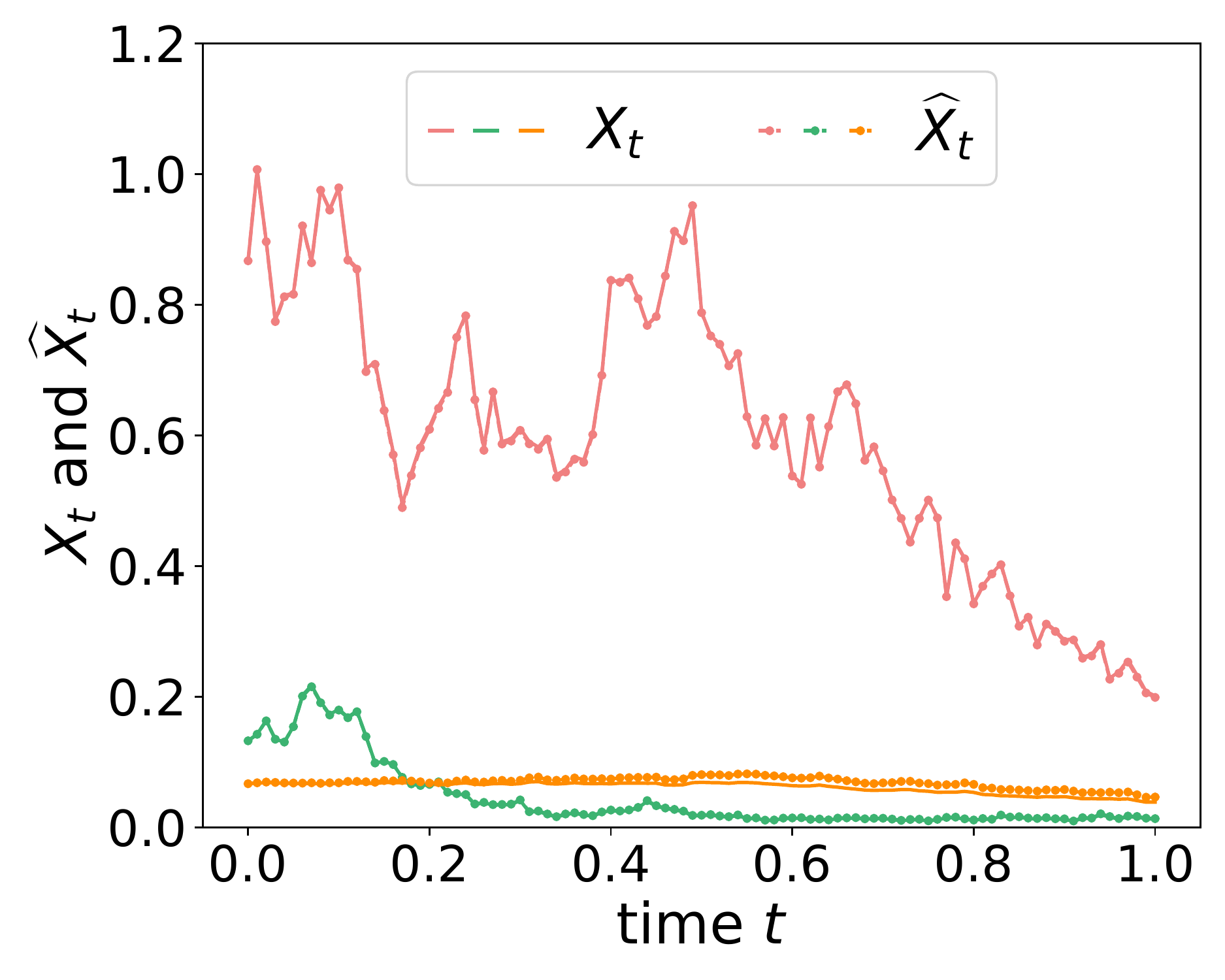}
     }
     \subfloat[$m_t =  \exp\EE(\log X_t |\mcF^B_t)$]{
         \includegraphics[width=0.6\columnwidth]{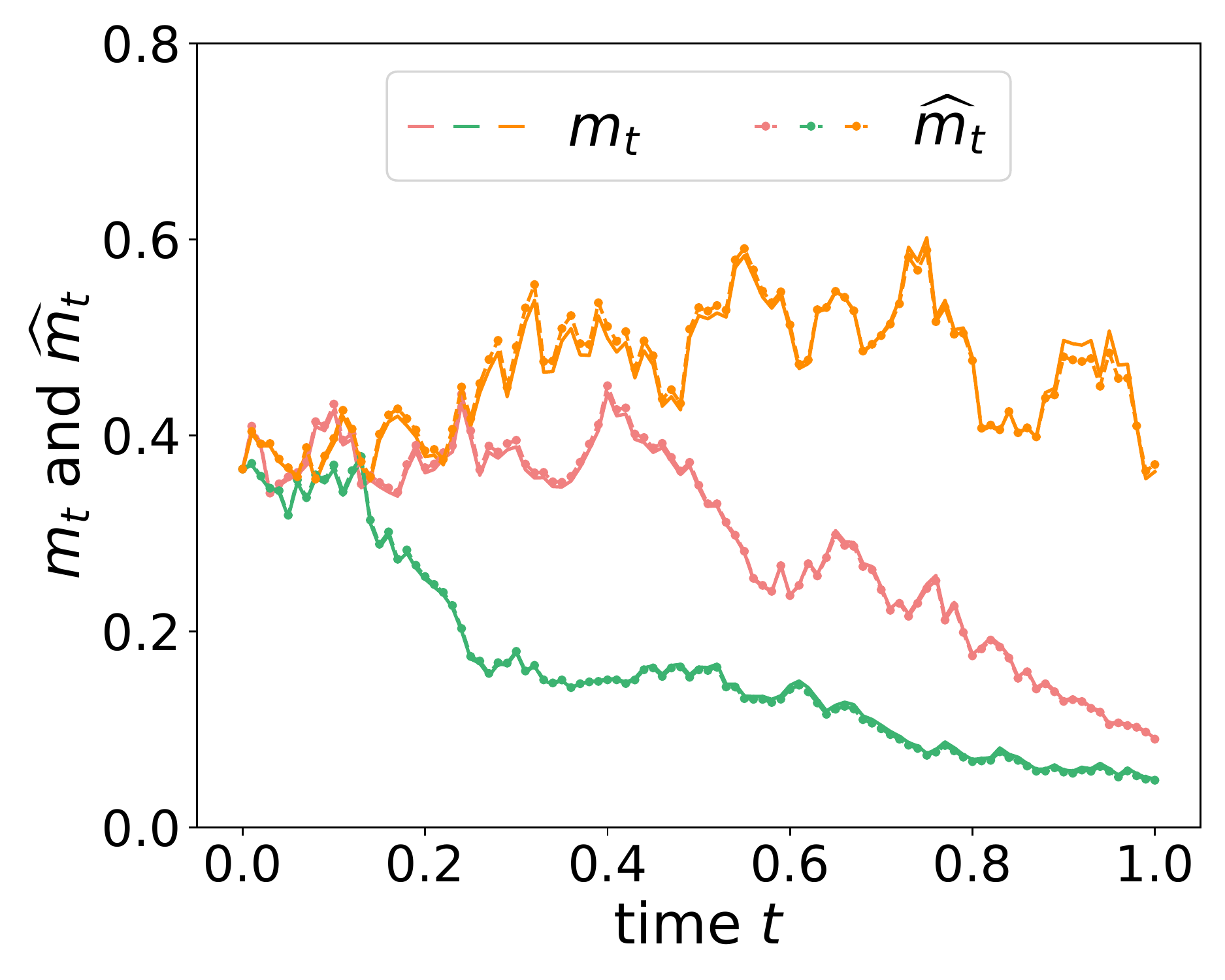}
     }
     \subfloat[Maximized Utility]{
         \includegraphics[width=0.6\columnwidth]{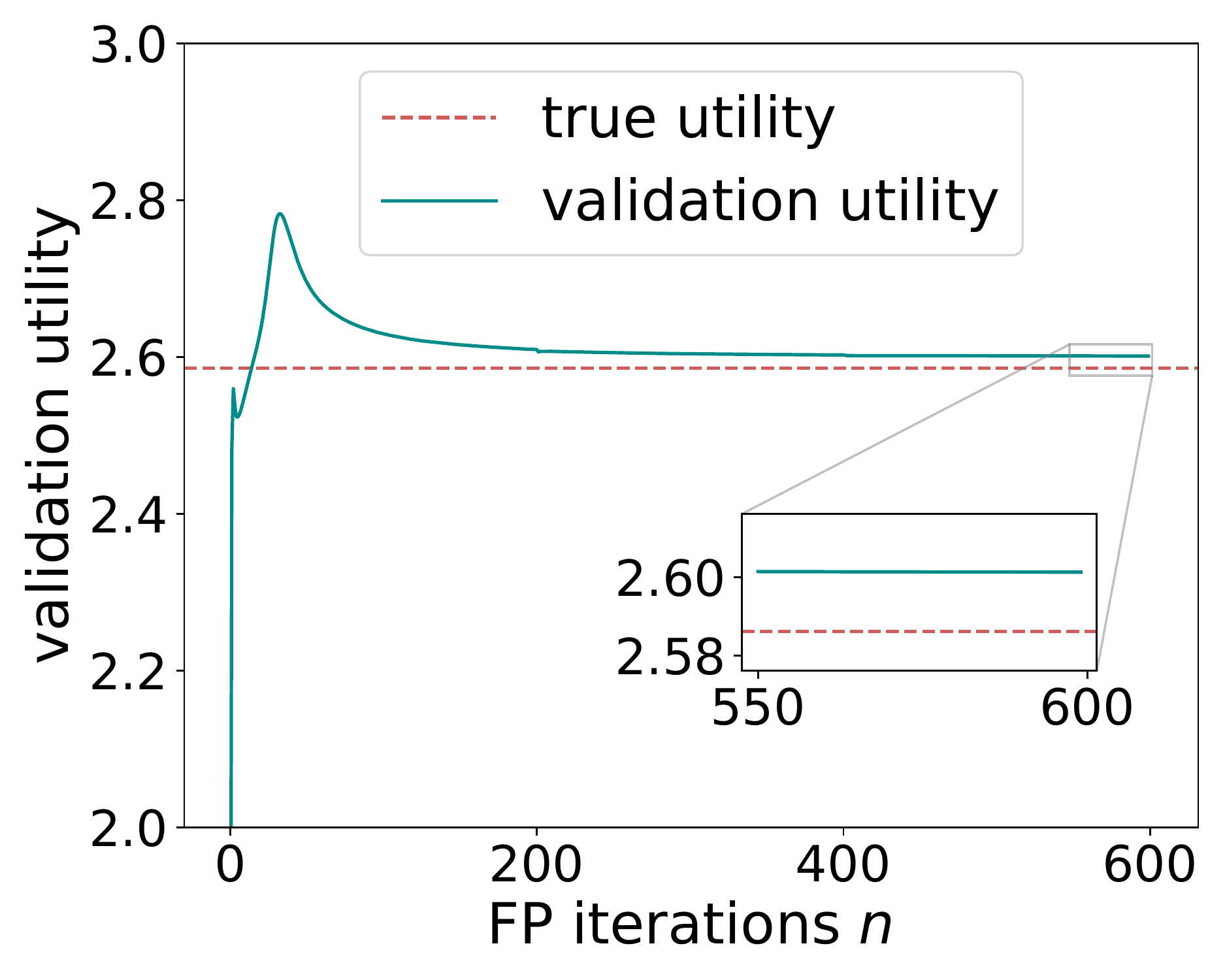}}
    
    \caption{Panels (a) and (b) give three trajectories of $X_t$ and $m_t=\exp\EE(\log X_t |\mcF^B_t)$ (solid lines) and their approximation (dashed lines) using different $(X_0, W, B)$ from test data. Panel (c) shows the maximized utility computed using validation data over fictitious play iterations. Parameter choices are: $\delta \sim U(2, 2.5), \mu\sim U(0.25, 0.35), \nu\sim U(0.2, 0.4), \theta, \xi \sim U(0,1), \sigma\sim U(0.2, 0.4)$, $\epsilon\sim U(0.5, 1)$.}
    \label{fig:OCI}
\end{figure*}
%\end{onecolumn}

%Test data for three different realizations $(X_0^i, W^i, B^i, \zeta^i)$.

%\begin{figure*}[ht]
%    \centering
%    \subfloat[$X_t$]{
%         \includegraphics[width=0.6\columnwidth]{InvestConsumption/SDE.pdf}
%     }
%     \subfloat[$m_t =  \exp\EE(\log X_t |\mcF^B_t)$]{
%         \includegraphics[width=0.6\columnwidth]{InvestConsumption/Xbar.pdf}
%     }
%     \subfloat[$\pi_t$]{
%         \includegraphics[width=0.6\columnwidth]{InvestConsumption/pi.pdf}} \\
%    \subfloat[$c_t$]{
%         \includegraphics[width=0.6\columnwidth]{InvestConsumption/c.pdf}}
%     \subfloat[$\Gamma_t =  \exp\EE(\log c_t |\mcF^B_t)$]{
%         \includegraphics[width=0.6\columnwidth]{InvestConsumption/Gamma.pdf}}
%     \subfloat[Maximized Utility]{
%         \includegraphics[width=0.6\columnwidth]{InvestConsumption/valid_util.pdf}}
    
%    \caption{Plots on test data for three different $(X_0^i, W^i, B^i, \zeta^i)$. Solid line is the benchmark solution and dashed line is the predicted result. (a), (b), (c), (d) and (e) plot three trajectories of $X_t$, $m_t=\exp\EE(\log X_t |\mcF^B_t)$, $\pi_t$, $c_t$, $\Gamma_t=\exp\EE(\log c_t |\mcF^B_t)$ (solid lines) and their approximation (dashed lines) using different $(X_0, W, B)$ from test data. Panel (f) depicts the maximized utility computed using validation data over fictitious play iterations. Parameter choises are: $\delta \sim U(2, 2.5), \mu\sim U(0.25, 0.35), \nu\sim U(0.2, 0.4), \theta\sim U(0,1), \sigma\sim U(0.2, 0.4)$, $\epsilon\sim U(0.5, 1)$.}
%    \label{fig:InvestConsump}
%\end{figure*}

{\it Comparison with the nested algorithm.} We run both Sig-DFP and the nested algorithm for the training data size of (INP, CNP)$=(2^4,2^4),\;(2^6,2^6)\;(2^8,2^8)$, where INP means the number of individual noise paths and CNP means the number of common noise paths. From the comparisons of running time, memory, and relative $L^2$ errors in Tables~\ref{tab:RTMem} and \ref{tab:L2err}, one can see that the accuracy is mainly affected by the size of (INP, CNP) used for training the neural network. Sig-DFP has the advantage of reducing memory request and running time, which allows it to use a larger size of data, {\it e.g.}, (INP, CNP)$=(2^{15},2^{15})$, to produce much better accuracy. The quadratic growth of memory in the nested algorithm, evidenced by the first three columns of data in Tables~\ref{tab:RTMem} (least squares growth rate $\approx 2$), makes us unable to run the nested algorithm beyond $(2^8,2^8)$ in our current computing environment due to its high demand for memory.

\begin{table*}[htb!]
\vspace{-1.2em}
\caption{Running time (hours) and Memory (GBs) comparisons between Sig-DFP and the nested algorithm for different (INP, CNP)'s. INP $=$ \# of individual noise paths, CNP $=$ \# of common noise paths, and NA $=$ Not Available due to high demand for memory.}
\label{tab:RTMem}
\vskip 0.1in
%\vspace{-1em}
    \begin{center}
    \begin{small}
    \begin{sc}
    \begin{tabular}{lcccccr}
    \toprule
         (INP, CNP) &  $(2^4,2^4)$ & $(2^6,2^6)$ & $(2^8,2^8)$ & $(2^{12},2^{12})$ & $(2^{15},2^{15})$\\
         \midrule
        Nested Algorithm & $(0.09, 2.1)$  & $(0.46, 4.1)$ & $(4.3, 43.5)$  & NA & NA \\
             %\midrule
        Sig-DFP & $(0.09, 1.9)$  & $(0.1, 2.0)$ & $(0.17, 2.3)$  & $(0.33, 4.8)$ & $(1.3, 27)$ \\
        \bottomrule
    \end{tabular}
    \end{sc}
    \end{small}
    \end{center}
    \vspace{-1.5em}
\end{table*}

\begin{table*}[htb!]
\vspace{-0em}
\caption{The comparisons of relative $L^2$ errors on $(\pi, c)$ between Sig-DFP and the nested algorithm for different (INP, CNP)'s. INP $=$ \# of individual noise paths, CNP $=$ \# of common noise paths, and NA $=$ Not Available due to high demand for memory.}
\label{tab:L2err}
\vskip 0.1in
%\vspace{-1em}
    \begin{center}
    \begin{small}
    \begin{sc}
    \begin{tabular}{lcccccr}
    \toprule
         (INP, CNP) &  $(2^4,2^4)$ & $(2^6,2^6)$ & $(2^8,2^8)$ & $(2^{12},2^{12})$ & $(2^{15},2^{15})$\\
         \midrule
        Nested Algorithm & $(53\%, 44\%)$  & $(36\%, 41\%)$ & $(79.4\%, 16.2\%)$  & NA & NA \\
             %\midrule
        Sig-DFP & $ (85.8\%, 48.1\%)$  & $(43.3\%, 44.9\%)$ & $(49\%, 43\%)$  & $(18\%, 38\%)$ & $(11\%, 6\%)$ \\
        \bottomrule
    \end{tabular}
    \end{sc}
    \end{small}
    \end{center}
    \vspace{-1.5em}
\end{table*}

\begin{table*}[htb!]
\vspace{-0em}
\caption{The comparisons of running time (hours) for different signature depth $M$ and dimension $n_0$ using (INP, CNP)$\small =(2^{15},2^{15})$ .}
\label{tab:depth}
\vskip 0.1in
%\vspace{-1em}
    \begin{center}
    \begin{small}
    \begin{sc}
    \begin{tabular}{lccccccccr}
    \toprule
          ($n_0$, Depth $M$) &  $(1,1)$ &  $(1,2)$ & $(1,3)$ & $(1,4)$ & $(5,1)$ & $(5,2)$ & $(5,3)$ &  $(5,4)$ \\
         \midrule
        Running Time (hours) & $1.2$ & $1.2$  & $1.2$ & $1.3$ & $1.2$ & $1.3$ & $1.5$ & $2.6$ \\
        \bottomrule
    \end{tabular}
    \end{sc}
    \end{small}
    \end{center}
    \vspace{-1.5em}
\end{table*}

{\it Comparisons of running time for different signature depth $M$ and dimension $n_0$.} We choose the data size (INP, CNP)$=(2^{15},2^{15})$ and compare the running time for different $(n_0, M)$'s in Table~\ref{tab:depth}. Choosing $M=1, 2, 3, 4$ yield the relative $L^2$ errors of controls $(\pi, c)$ as $(15.9\%, 9.5\%)$, $(11.4\%, 6.3\%)$, $(11.4\%, 6.3\%)$ and $(11.3\%, 6.1\%)$ for $n_0=1$, respectively. Note that, compared to $M=1$, taking $M=2$ improves the accuracy significantly but not $M=3,4$. This is because the curves of $\log(c_t)$ and $\log(X_t^*)$ are approximately either linear or quadratic in $t$, as shown in Figure~\ref{fig:OCIcurves} in Appendix~\ref{app:OCI} after taking a logarithm, which implies that the signatures of depth $M=2$ will be sufficient to produce good accuracy. We remark that Sig-DFP has no difficulty computing high-dimensional problems, evidenced by the running time of $n_0=5$ cases in Table~\ref{tab:depth}. We focus on one-dimensional problems since, to our best knowledge, the closed-form non-trivial solutions only exist in one-dimensional cases, which can serve as the benchmark solutions. More details about $n_0=5$ are given in Appendix~\ref{app:highd}.

\section{Conclusion}

In this paper, we propose a novel single-loop algorithm, named signatured deep fictitious play, for solving mean-field games (MFGs) with common noise. We incorporate signature from rough path theory into the strategy of deep fictitious play \cite{Hu2:19,HaHu:19,han2020convergence}, and avoid the nested-loop structure in existing machine learning methods, which reduces the computational cost significantly. Analysis of the complexity and convergence for the proposed algorithm is provided. The effectiveness of the algorithm is justified by three applications, and in particular, we report the first deep learning work to deal with extended MFGs with common noise. In the future, we shall study deep learning algorithms for MFGs with common noise in more general settings \cite{HuZa:21}.

\section*{Acknowledgement}
R.H.~was partially supported by NSF grant DMS-1953035, the Faculty Career Development Award and the Research Assistance Program Award, University of California, Santa Barbara. M.M. and R.H. are grateful to the reviewers for their valuable and constructive comments.

\newpage
\bibliography{citations}
\bibliographystyle{icml2021}

%\pagebreak
\newpage
\quad 
\newpage
\appendix

\begin{onecolumn}
\section{Preliminaries on Rough Path Theory and Signatures}\label{app:signature}
%Signature is an special object from rough paths theory. We show the factorial decay property by the definition and extension theorem of rough paths.
%We will follow \cite{friz_victoir_2010},\cite{lyons2002system} and \cite{LyonsTerryJ2007DEDb}. Let $\Delta_T$ denote the simplex $\{(s,t)\in[0,T]^2: 0\le s\le t\le T\}$ and the truncated tensor algebra $T^n(\RR^d)=\bigoplus_{k=0}^n (\RR^d)^{\bigotimes k}$ . Recall the definition of multiplicative functional.

%Signature is a concept in rough path theory, representing and providing possible reconstruction of a rough path. 

In this appendix, we shall follow \citet{lyons2002system, LyonsTerryJ2007DEDb,friz_victoir_2010} and briefly introduce rough path theory and signatures. We will also give the proof of Lemma~\ref{lemma:propbysig} using the factorial decay property of signatures.
%the proof details for the factorial decay property of signatures (Lemma~\ref{lemma:propbysig}). 
Denote by $\Delta_T$ the simplex $\{(s,t)\in[0,T]^2: 0\le s\le t\le T\}$, and by $T^n(\RR^d)=\bigoplus_{k=0}^n (\RR^d)^{\bigotimes k}$ the truncated tensor algebra. 

%Recall the definition of multiplicative functional.

%We show the factorial decay property by the definition and extension theorem of rough paths.
%We will follow . 
\begin{defn}[Multiplicative Functional]
Let $\bX: \Delta_T \to T^{n}(\RR^d)$, with $n\geq 1$ as an integer. For each $(s,t)\in \Delta_T$, $\bX_{s,t}$  denotes the image of $(s,t)$ under the mapping $\bX$, and we write $$\bX_{s,t} = (\bX_{s,t}^0, \bX_{s,t}^1, \dots, \bX_{s,t}^n)\in T^n(\RR^d).$$
The function $\bX$ is called a multiplicative functional of degree $n$ in $\RR^d$ if $\bX_{s,t}^0=1$ for all $(s,t)\in \Delta_T$ and
	\begin{equation}
	\label{chens_identity}
	\bX_{s,u} \otimes \bX_{u,t} = \bX_{s,t}, \,\, \forall s,u,t\in [0,T], \,\, s\le u\le t,
	\end{equation}
	which is called Chen's identity.
\end{defn}
Rough paths will be defined as a multiplicative functional with extra regularization conditions.
%\eqref{chens_identity} is also called Chen's identity. 

\begin{defn}[Control]
A control function on $[0,T]$ is a continuous non-negative function $\omega$ on the simplex $\Delta_T$ which is supper-additive in the sense that
$$\omega(s,u) + \omega(u,t) \le \omega(s,t) \,\,\,\, \forall s\le u\le t \in [0,T].$$
\end{defn}

It is easy to see that $\omega(t,t)=0$ for any control $\omega$. In the following, we use the notation $x! = \Gamma(x+1)$, where $\Gamma(\cdot)$ is the Gamma function and x is a positive real number.

\begin{defn}
Let $p\ge 1$ be a real number and $n\ge 1$ be an integer. Denote $\omega: \Delta_T \to [0, +\infty)$ as a control and $\bX:\Delta_T\to T^n(\RR^d)$ as a multiplicative functional. Then we say that $\bX$ has finite $p$-variation on $\Delta_T$ controlled by $\omega$ if
\begin{equation}
\label{control_eq}
	\| \bX^i_{s,t} \| \le \frac{ \omega(s,t)^{\frac{i}{p}} }{\beta (\frac{i}{p})!} \,\,\,\, \forall i=1,\dots,n, \,\,\,\, \forall(s,t)\in\Delta_T,
\end{equation}
where $\|\cdot\|$ is the tensor norm induced by the norm on $\RR^d$.
We will call that $\bX$ has finite $p$-variation in short if there exists a control $\omega$ such that \eqref{control_eq} is satisfied.
\end{defn}

Note that in \eqref{control_eq}, $\beta$ is a constant depending only on $p$. We are now ready to define the rough paths.

\begin{defn}[Rough Path]
\label{defn_roughpath}
Let $p\ge 1$ be a real number. A $p$-rough path in $\RR^d$ is a multiplicative functional of degree $\lfloor p \rfloor$ with finite $p$-variation. The space of $p$-rough paths is denoted by $\Omega_p(\RR^d)$.
\end{defn}
Given a continuous path $X:[0,T]\to \RR^d$ with bounded $p$-variation, one can construct a $\lfloor p \rfloor$-rough path $\bX$ with $\bX^1_{s,t} = X_t - X_s$ for any  $s\le t$. In particular, truncated siganture $S^{\lfloor p \rfloor}(X)\in T^{\lfloor p \rfloor}(\RR^d)$ is a $p$-rough path. The following fundamental theorem of rough paths allows us to make extension of a $p$-rough path, 

\begin{thm}[Extension Theorem, \citet{lyons2002system}] 
\label{extension_thm} %\footnote{Cite the original theorem. (TI)}
Let $p\ge 1$ be a real number and $n\ge 1$ an integer. Denote $\bX: \Delta_T \to T^n(\RR^d)$ as a multiplicative functional with finite $p$-variation controlled be a control $\omega$. Assume that $n\ge\lfloor p \rfloor$, then there exists a unique extension of $\bX$ to a multiplicative functional $\Delta_T\to T((\RR^d))$ which possesses finite $p$-variation.

More precisely, for every $m\ge \lfloor p\rfloor + 1$, there exists a unique continuous function $\bX^m:\Delta_T\to (\RR^d)^{\bigotimes m}$ such that
$$(s,t) \to \bX_{s,t}=\left( 1, \bX^1_{s,t}, \dots, \bX^{\lfloor p\rfloor}_{s,t}, \dots, \bX^m_{s,t}, \dots \right) \in T((\RR^d))$$
is a multiplicative functional with finite $p$-variation controlled by $\omega$. By this we mean that
\begin{equation}
\label{extension_eq}
	\|\bX^i_{s,t}\| \le \frac{\omega(s,t)^{\frac{i}{p}}}{\beta (\frac{i}{p})!} \,\,\,\, \forall i\ge 1, \,\,\,\, \forall (s,t)\in \Delta_T.
\end{equation}
%In particular, the $\beta$ in \eqref{control_eq} and \eqref{extension_eq} are the same and given by 
%$$\beta = p^2 \left( 1 + \sum_{r=3}^\infty (\frac{2}{r-2})^{\frac{\lfloor p\rfloor+1}{p}} \right).$$
\end{thm}
Signature can be seen as an extension of rough path, and its factorial decay property follows by \eqref{extension_eq}. The control function is related to $p$-variation of path. Given that $x\in \mathcal{V}^p([0,T], \RR^d)$, $S^{\lfloor p \rfloor}(x)$ is a $p$-rough path and one candidate for its control function is
\begin{equation}
\label{eq:omega}
    \omega(s,t) = \sum_{i=1}^{\lfloor p \rfloor} \sup_{D\subset [s,t]}\sum_{k} \| x^i_{t_{k+1}} -x^i_{t_{k}}\|^{p/i},
\end{equation}
where the norm is the tensor norm induced by Euclidean norm in $\RR^d$.

Let $S^{\lfloor p \rfloor}(\Omega_1) = \{S^{\lfloor p \rfloor}(x): x\in\Omega_1(\RR^d)\}$, and $\bY$ be a $p$-rough path. We call $\bY$ a $p$-geometric rough path if $\bY$ is in the closure of $S^{\lfloor p \rfloor}(\Omega_1)$ under $p$-variation metric, where $p$-variation metric is given by
\begin{equation}
    d_{p\text{-var}}(\bX,\bY) := \left(\sup_{D}\sum_{t_i\in D} \|\bX_{t_i,t_{i+1}}-\bY_{t_i,t_{i+1}}\|^p\right)^{1/p}, \quad\quad \bX,\bY\in \Omega_p(\RR^d).
\end{equation}

\begin{proof}[Proof of Lemma \ref{lemma:propbysig}]
    By constructing the iterated integral in Stratonovich sense, $S(\hat{B}_{0:T})$ is the signature of a $p$-geometric rough path $\forall p\in(2,3)$ \cite{friz_victoir_2010}, and thus it characterizes $B_{0:T}$ uniquely. Therefore, conditional distribution $\mu_t = \EE[\iota(X_t) \vert \mcF_t^B]$ can be written as $\mu_t := \mu(t, B_{0,t})=\mu(\hat{B}_{0,t})$.
    
    By Theorem \ref{thm:sig_universality}, for any $\eps >0$ there exits $l$ such that
    \begin{equation}
        \sup_{\hat{B}\in K} |\mu(\hat{B}_{0:T}) - \langle l, S(\hat{B}_{0:T}) \rangle| <\frac{\epsilon}{2}.
    \end{equation}
    Since $|\langle l, S(\hat{B}_{0:T})- S^M(\hat{B}_{0:T})\rangle | \le \|l\| \cdot \|S(\hat{B}_{0:T})- S^M(\hat{B}_{0:T})\|$ where the first norm is functional norm and second is tensor norm and $\|S(\hat{B}_{0:T})- S^M(\hat{B}_{0:T})\| = \sum_{i\ge M+1} \|\hat{B}_{0:T}^{i}\|$. By the compactness of $K$, and \eqref{extension_eq}, \eqref{eq:omega}, 
    $\sum_{i\ge M+1} \|\hat{B}_{0:T}^{i}\|$ admits a convergent uniform norm over $\hat{B}\in K$ and goes to $0$ as $M\to\infty$. Then for $M$ large enough,
    \begin{equation}
        \sup_{\hat{B}\in K} |\mu(\hat{B}_{0:T}) - \langle l, S^M(\hat{B}_{0:T}) \rangle|         < \frac{\epsilon}{2} + \sup_{\hat{B}\in K}|\langle l, S(\hat{B}_{0:T})- S^M(\hat{B}_{0:T})\rangle | 
        <\frac{\epsilon}{2} + \frac{\epsilon}{2} = \epsilon.
        \label{eq:universalatT}
    \end{equation}
    For $t<T$, we extend path $\hat{B}_{0:t}$ to space $\mathcal{V}^p([0,T], \RR^d)$ by defining
    $$\Tilde{B}^t_s := \left\{\begin{array}{cc}
     \hat{B}_s,    & 0\le s\le t  \\
      \hat{B}_t,   & t < s \le T.
    \end{array} \right.$$
    Then $\Tilde{B}^t_{0:T}\in \mathcal{V}^p([0,T], \RR^d)$, $S(\Tilde{B}^t_{0:T}) = S(\hat{B}_{0:t})$ by Chen's identity \eqref{chens_identity}, and $\mu(\hat{B}_{0:t}) = \mu(\Tilde{B}^t_{0,T})$. Denote $\tilde{K} = \{\Tilde{B}^t_{0:T}, \forall t\in[0,T]: \tilde B_{0:T}^t \text{ is constructed by } \hat B_{0:t} \text{ and } \hat B\in K\}$. Thus $\tilde{K}$ is also compact.
    \begin{align}
        \sup_{t\in[0,T]}\sup_{\hat{B}\in K}  |\mu(\hat{B}_{0:t}) - \langle l, S^M(\hat{B}_{0:t}) \rangle| &=  \sup_{t\in[0,T]}\sup_{\hat{B}\in K}  |\mu(\Tilde{B}^t_{0:T}) - \langle l, S^M(\Tilde{B}^t_{0:T}) \rangle| \nonumber\\
        &= \sup_{\Tilde{B}\in \tilde{K}}  |\mu(\Tilde{B}_{0:T}) - \langle l, S^M(\Tilde{B}_{0:T}) \rangle| < \epsilon,
    \end{align}
    where the second equality is due to the construction of $\tilde B_{0:T}^t$ and the last inequality is by \eqref{eq:universalatT}.
\end{proof}

%\revise{
%We should emphasize that by saying `the first several terms contain most of the information of a path', it means that $\mu(\hat{B}_{0:T})$ can be approximated very well by the linear functional on truncated signatures, see \eqref{eq:universalatT}. The signature depth $M$ has not to be large due to the factorial decay property stated in \eqref{extension_eq}.
%}

\section{Details of Implementing the Sig-DFP Algorithm}\label{app:SigDFP}
%Denote by $\mcN$ the class of fully connected neural network and $\alpha_\varphi\in \mcN$. 

The simulation of $X^{i, (n)}$ and $J_B(\varphi, \hat{\mu}^{(n-1)})$ follows
\begin{align}
    & J_B(\varphi, \hat{\mu}^{(n-1)}) = \frac{1}{B}\sum_{i=1}^B \left(\sum_{k=0}^{L-1}f(t_k, X_k^{i, (n)}, \hat{\mu}^{(n-1)}_k(\omega^i), \alpha_\varphi(t_k, X^{i, (n)}_{k}, \hat{\mu}^{(n-1)}_k(\omega^i)) \Delta_k + g(X_L, \hat{\mu}^{(n-1)}_L(\omega^i)) \right), \label{def:J_algo} \\
    & X^{i, (n)}_{k+1} = X^{i, (n)}_{k} + b(t_k, X^{i,(n)}_k, \hat{\mu}^{(n-1)}_k(\omega^i), \alpha_\varphi(t_k, X^{i, (n)}_{k}, \hat{\mu}^{(n-1)}_k(\omega^i)) \Delta_k \nonumber\\
    & \quad\quad\quad\quad\quad\quad\,\,\, + \sigma(t_k, X^{i, (n)}_{k}, \hat{\mu}^{(n-1)}_k(\omega^i), \alpha_\varphi(t_k, X^{i, (n)}_{k}, \hat{\mu}^{(n-1)}_k(\omega^i)) \Delta W^i_k \nonumber \\
    & \quad\quad\quad\quad\quad\quad\,\,\, + \sigma^0(t_k, X^{i, (n)}_{k}, \hat{\mu}^{(n-1)}_k(\omega^i), \alpha_\varphi(t_k, X^{i, (n)}_{k}, \hat{\mu}^{(n-1)}_k(\omega^i)) \Delta B^i_k, \quad X^{i, (n)}_0 = X^i_0 \sim \mu_0, \label{def:Xt_algo}
\end{align}
where $\hat{\mu}^{(n-1)}_k(\omega^i)$ is computed by  $\hat{\mu}^{(n-1)}_k(\omega^i) = \langle \Bar{l}^{(n-1)}, S^M(\hat{B}^i_{0:t_k}) \rangle$ with $\Bar{l}^{(n-1)}$ obtained from the previous round of fictitious play. Then $l^{(n)}$ is calculated by regressing $\{\iota(X^{i, (n)}_0), \iota(X^{i, (n)}_{L/2}), \iota(X^{i, (n)}_L)\}_{i=1}^N$ on $\{S^M(\hat{B}_{0,0}), S^M(\hat{B}_{0,t_{L/2}}), S^M(\hat{B}_{0,t_L})\}_{i=1}^N$, and we update $\bar{l}^{(n)} = \frac{n-1}{n}\bar{l}^{(n-1)} +\frac{1}{n}l^{(n)}$ for $n\ge 1$. The algorithm starts with a random initialization $\bar l^{(0)}$ to produce  $\hat{\mu}^{(0)}$. 

{\bf Linear-Quadratic MFGs.} We set $\alpha_\varphi$ to be a feed-forward NN with two hidden layers of width 64. The signature depth is chosen at $M=2$. This model is trained for $N_{round}=500$ iterations of fictitious play. Note that fictitious play has a slow convergence speed since our initial guess $m^{(0)}$ is far from the truth. Therefore, we only apply averaging over distributions (or linear functions) during the second half iteration. We set the learning rate as $0.1$ for the first half iterations and $0.01$ for the second half. The minibatch size is $B=2^{10}$, and hence $N_{batch}=2^5$.

{\bf Mean-field Portfolio Game.} We consider signature depth $M=2$ and use a fully connected neural network $\pi_\varphi$ with four hidden layers to estimate $\pi_t$. Since different players are characterized by their type vectors $\zeta$, $\pi_\varphi$ takes $(\zeta, t, X_t, m_t)$ as inputs. Hidden neurons in each layer are (64, 32, 32, 16). We train our model with $N_{round}=500$ rounds fictitious play. The learning rate starts at $0.1$ and is reduced by a factor of $5$ after every 200 rounds. The minibatch size is $B=2^{10}$, and hence $N_{batch}=2^5$.

{\bf Mean-field Game of Optimal Consumption and Investment.} In this example, signature depth is $M=4$. The optimal controls $(\pi_t, c_t)_{0\le t\le 1}$ are estimated by two neural networks $\pi_\varphi$ and $c_\varphi$, each with three hidden layers. Due the nature of heterogeneous extended MFG, both $\alpha_\varphi$ and $c_\varphi$ take $(\zeta_t, t, X_t, m_t, \Gamma_t)$ as the inputs. Hidden layers in each network have width (64, 64, 64). We will propagate two conditional distribution flows, {\it i.e.}, two linear functionals $\Bar{l}^{(n)}, \Bar{l}_c^{(n)}$ during each round fictitious play. Instead of estimating $m_t, \Gamma_t$ directly, we estimate $\EE[\log X_t^\ast |\mcF^B_t], \EE[\log c^*_t|\mcF^B_t]$ by $\langle \Bar{l}^{(n)}, S^4(\hat B_{0:t}) \rangle$, $\langle \Bar{l}_c^{(n)}, S^4(\hat B_{0:t}) \rangle$, and then take the exponential to get $m_t, \Gamma_t$. To ensure the non-negativity condition, we evolve $\log X_t$ according to \eqref{def: InvestConsumplogSDE}, use $c_\varphi$ to predicted $\log c_t$, and then take exponential to get $c_t, X_t$. We use $N_{round}=600$ rounds fictitious play training, learning rate 0.1 decaying by a factor of 5 for every 200 rounds, the minibatch size $B=2^{11}$, and hence $N_{batch}=2^4$. 

The training time for all three experiments with sample size $N=2^{13}, 2^{14}, 2^{15}$ is given in Table \ref{tab:training time}.

\vspace{-1em}
\begin{table}[ht]
\caption{Training time in minutes. Here LQ-MFG $=$ Linear-Quadratic mean-field games, MF Portfolio $=$ Mean-field Portfolio Game, and MFG with Consump. $=$  Mean-field Game of Optimal Consumption and Investment. }
\label{tab:training time}
\vskip 0.15in
    \begin{center}
    \begin{small}
    \begin{sc}
    \begin{tabular}{lcccr}
    \toprule
         & $N=2^{13}$ & $N=2^{14}$ & $N=2^{15}$\\
         \midrule
        LQ-MFG & $12.4$  & $23.7$ & $46.7$ \\
        MF Portfolio & $12.3$  & $23.3$ & $45.5$ \\
        MFG with Consump. & $23.4$  & $40.9$ & $80.1$ \\
        \bottomrule
    \end{tabular}
    \end{sc}
    \end{small}
    \end{center}
\end{table}

\section{Proof of Theorems~\ref{thm:cvg} and \ref{thm:numcvg}}\label{app:algorithm}
\newcommand{\pd}{\partial}
We first list all main assumptions on $(b, \sigma, \sigma^0, f, g)$ that will be used to prove Theorem~\ref{thm:cvg}. Let $\|\cdot\|$ be the Euclidean norm and $K$ be the same constant for all assumptions below.
\begin{assump}\label{assump:cvg} We make assumptions {\bf A1-A3} and {\bf B1-B3} as follows.
\vspace{-1em}
\begin{itemize}
%    \item[]
    \item[\bf A1.]
    (Lipschitz)
    $\pd_x f, \pd_\alpha f, \pd_x g$ exist and are $K$-Lipschitz continuous in $(x, \alpha)$ uniformly in $(t, \mu)$, {\it i.e.}, for any $t\in [0,T]$, $x, x^\prime\in\RR^d, \alpha, \alpha^\prime\in\RR^m, \mu\in \mcP^2(\RR^d)$,
    \begin{align*}%\label{def:lipschitiz}
       \|\pd_x g(x, \mu)-\pd_x g(x^\prime, \mu)\|&\le K\|x-x^\prime\|,\nonumber \\
         \|\pd_x f(t,x,\mu,\alpha)-\pd_x f(t,x^\prime, \mu, \alpha^\prime)\|& \le K(\|x-x^\prime\|+\|\alpha - \alpha^\prime\|), \\
         \|\pd_\alpha f(t,x,\mu,\alpha)-\pd_\alpha f(t,x^\prime, \mu, \alpha^\prime)\| &\le K(\|x-x^\prime\|+\|\alpha - \alpha^\prime\|). \nonumber
    \end{align*}
The drift coefficient $b(t, x, \mu, \alpha)$ in \eqref{def:Xt} takes the form 
\begin{equation*}
    b(t, x, \mu, \alpha) = b_0(t, \mu) + b_1(t) x + b_2(t)\alpha,
\end{equation*}
where $b_0 \in \RR^d$, $b_1 \in \RR^{d \times d}$ and $b_2 \in \RR^{d \times m}$ are measurable functions and bounded by $K$. The diffusion coefficients $\sigma(t, x, \mu)$ and $\sigma^0(t, x, \mu)$ are uncontrolled and $K$-Lipschitz in $x$ uniformly in $(t,\mu)$:
\begin{equation*}
    \|\sigma(t,x, \mu)\| \leq K\|x - x'\|, \quad \|\sigma^0(t,x, \mu)\| \leq K\|x - x'\|.
\end{equation*}

\item[\bf A2.]
(Growth)
$\pd_x f, \pd_\alpha f, \pd_x g$ satisfy a linear growth condition, {\it i.e.}, for any $t\in[0,T]$, $x\in\RR^d, \alpha\in\RR^m, \mu\in\mcP^2(\RR^d)$,
    \begin{align*}
        & \quad\quad \|\pd_xg(x, \mu)\|\le K\bigg( 1 + \|x\| + \left(\int_{\RR^d} \|y\|^2 \ud \mu(y)\right)^{\frac{1}{2}} \bigg),\nonumber \\
        & \|\pd_x f(t,x,\mu,\alpha)\| \le K\bigg( 1 + \|x\|  + \|\alpha\|+ \left(\int_{\RR^d}\|y\|^2 \ud \mu(y)\right)^{\frac{1}{2}} \bigg), \\
        & \|\pd_\alpha f(t,x,\mu,\alpha)\| \le K\bigg( 1 + \|x\| + \|\alpha\| + \left(\int_{\RR^d}\|y\|^2 \ud \mu(y)\right)^{\frac{1}{2}} \bigg). \nonumber
    \end{align*}
    In addition $f, g$ satisfy a quadratic growth condition in $\mu$:
    \begin{align*}
        & \quad |g(0, \mu)| \le K\bigg(1+\int_{\RR^d}\|y\|^2 \ud \mu(y)\bigg), \nonumber \\
        &|f(t, 0, \mu, 0)| \le K\bigg(1+\int_{\RR^d} \|y\|^2 \ud \mu(y)\bigg).
    \end{align*}

\item[\bf A3.] (Convexity)
    $g$ is convex in $x$ and $f$ is convex jointly in $(x, \alpha)$ with strict convexity in $\alpha$, {\it i.e.}, for any $x, x^\prime\in\RR^d, \mu\in\mcP^2(\RR^d)$,
    \begin{equation*}(\pd_x g(x,\mu)-\pd_x g(x^\prime, \mu))^T(x-x^\prime)\ge 0,\end{equation*} 
    and there exist a constant $c_f>0$ such that for any $t\in[0,T]$, $x,x^\prime\in\RR^d,\alpha,\alpha^\prime\in\RR^m$, $\mu\in\mcP^2(\RR^d)$,
    \begin{equation*}
        f(t,x^\prime, \alpha^\prime, \mu)\ge f(t,x,\alpha,\mu) + \pd_x f(t,x,\alpha, \mu)^T(x^\prime-x)+\pd_\alpha f(t,x,\alpha, \mu)^T(\alpha^\prime-\alpha)+c_f\|\alpha^\prime-\alpha\|^2.
    \end{equation*}

\item[\bf B1.] (Lipschitz in $\mu$)
    $\pd_x g, \pd_x f, \pd_\alpha f, b_0, \sigma, \sigma^0$ are Lipschitz continuous in $\mu$ uniformly in $(t,x)$, {\it i.e.}, there exists a constant $K$ such that
    \begin{align*}
        \|\pd_xg(x,\mu)-\pd_xg(x,\mu^\prime)\|&\le K \mc{W}_2(\mu,\mu^\prime), \nonumber\\
         \|\pd_xf(t,x,\mu,\alpha)-\pd_xf(t,x,\mu^\prime,\alpha)\|&\le K \mc{W}_2(\mu,\mu^\prime) \nonumber \\
          \|\pd_\alpha f(t,x,\mu,\alpha)-\pd_\alpha f(t,x,\mu^\prime,\alpha)\|&\le K \mc{W}_2(\mu,\mu^\prime)\\
          \|b_0(t, \mu) - b_0(t, \mu')\| &\leq K \mc{W}_2(\mu, \mu'),\\
             \|\sigma(t, x, \mu) - \sigma(t,x,  \mu')\| &\leq K \mc{W}_2(\mu, \mu'),\\
             \|\sigma^0(t, x, \mu) - \sigma^0(t, x, \mu')\| &\leq K \mc{W}_2(\mu, \mu'),
    \end{align*}
    for all $t\in[0,T], x\in\RR^d,\alpha\in\RR^m$, $\mu, \mu^\prime\in\mcP^2(\RR^d)$, where $\mc{W}_2$ is the 2-Wasserstein distance. 
    %This is equivalent to the following; for any $t\in[0,T]$, $x\in\RR^d,\alpha\in\RR^m, \xi, \xi^\prime\in \mc{L}^2(\bar{\Omega}, \bar{F}, \bar{\PP}, \RR^d)$ where $(\bar{\Omega}, \bar{F}, \bar{\PP})$ is arbitrary,
    %\begin{align}
    %    & \quad\quad \|\pd_x g(x, \mc{L}(\xi)) - \pd_x g(x, \mc{L}(\xi^\prime))\| \le K\|\xi-\xi^\prime\|_2, \nonumber \\
    %    & \|\pd_x f(t,x,\mc{L}(\xi), \alpha) - \pd_x f(t,x,\mc{L}(\xi^\prime), \alpha)\|\le K\|\xi-\xi^\prime\|_2
    %\end{align}
    %where $\|\cdot\|_2$ denote the $\mc{L}^2-norm$. 
   % The terms $(b_0, \sigma, \sigma^0)$ are Lipschitz continuous in $\mu$ uniformly in $(t,x)$:
    %\begin{align}
    %        & \|b_0(t, \mu) - b_0(t, \mu')\| \leq K \mc{W}_2(\mu, \mu'),\\
    %        & \|\sigma(t, x, \mu) - \sigma(t,x,  \mu')\| \leq K \mc{W}_2(\mu, \mu'),\\
    %        & \|\sigma^0(t, x, \mu) - \sigma^0(t, x, \mu')\| \leq K \mc{W}_2(\mu, \mu'),
    %\end{align}

\item[\bf B2.] (Separable in $\alpha, \mu$)
    $f$ is of the form
    \begin{equation*}
        f(t,x,\mu, \alpha) = f^0(t,x,\alpha)+f^1(t,x,\mu),
    \end{equation*}
    where $f^0$ is assumed to be convex in $(x,\alpha)$ and strictly convex in $\alpha$, and $f^1$ is assumed to be convex in $x$.

\item[\bf B3.] (Weak monotonicity)
For all $t\in[0,T]$, $\mu, \mu^\prime\in \mcP^2(\RR^d)$ and $\gamma \in \mcP^2(\RR^d\times\RR^d)$ with marginals $\mu, \mu^\prime$ respectively,
    \begin{align*}
    &\quad\quad \int_{\RR^d\times \RR^d} \big[(\pd_x g(x, \mu) - \pd_x g(y, \mu'))^T(x-y)\big]\gamma(\ud x, \ud y) \ge 0,\nonumber \\
    & \int_{\RR^d\times \RR^d} \big[(\pd_x f(t,x, \mu,\alpha) - \pd_x g(t, y, \mu', \alpha))^T(x-y)\big]\gamma(\ud x, \ud y) \ge 0.
    \end{align*}
%    Equivalently, for any $x\in\RR^d$, $\xi, \xi^\prime\in \mc{L}^2(\bar{\Omega}, \bar{F}, \bar{\PP}, \RR^d)$ where $(\bar{\Omega}, \bar{F}, \bar{\PP})$ is arbitrary,
%    \begin{align}
%        & \quad\quad \EE\bigg[\big(\pd_x g(\xi, \mc{L}(\xi)) - \pd_x g(\xi^\prime, \mc{L}(\xi^\prime))\big)^T(\xi-\xi^\prime)\bigg] \ge 0, \nonumber \\
%        &\EE\bigg[\big(\pd_x f(t, \xi, \mc{L}(\xi), \alpha) - \pd_x f(t, \xi^\prime, \mc{L}(\xi^\prime), \alpha)\big)^T(\xi-\xi^\prime)\bigg] \ge 0.
%    \end{align}
    \end{itemize}
\end{assump}

Note that Assumption \ref{assump:cvg} extends conditions \textbf{A} and \textbf{B} in \citet{ahuja2015mean} by considering general drift coefficient $b(t, x, \mu, \alpha)$ and non-constant diffusion coefficients $\sigma(t, x, \mu)$ and $\sigma^0(t, x, \mu)$.

Our proof of Theorem~\ref{thm:cvg} uses the probabilistic approach. To this end, we define the Hamiltonian by
\begin{equation*}
    H(t, x, y, \mu, \alpha) = b(t, x, \mu, \alpha) \cdot y + f(t, x, \mu, \alpha). 
\end{equation*}
Denote by $\hat \alpha$ the minimizer of the Hamiltonian which is unique due to Assumptions {\bf A1} and {\bf A3}:
\begin{equation}\label{def:alphahat}
    \hat\alpha(t, x, y, \mu) = \argmin_{\alpha\in \RR^m} H(t, x, y, \mu, \alpha).
\end{equation}
By the Lipschitz property of $\partial_\alpha f$ in $(t, \mu, \alpha)$ and the boundedness of $b_2(t)$, $\hat \alpha$ is Lipschitz in $(x, y, \mu)$. Let $\hat H$ be the Hamiltonian, with $\hat\alpha$ obtained in \eqref{def:alphahat},
\begin{equation}
    \hat H(t, x, y, \mu) = H(t, x, y, \mu, \hat\alpha(t, x, y, \mu)).
\end{equation}
Under Assumptions {\bf A1-A3}, with the stochastic maximum principle, the problem \eqref{def:J}-\eqref{def:Xt} is equivalent to solve the following FBSDE,  given $\mu \in \mc{M}([0,T]; \mcP^2(\RR^d))$,
\begin{equation}\label{def:FBSDE}
\begin{aligned}
    \ud X_t &= b(t, X_t, \mu_t, \hat\alpha(t, X_t, Y_t, \mu_t))\ud t + \sigma(t, X_t, \mu_t) \ud W_t + \sigma^0(t, X_t, \mu_t) \ud B_t, \quad X_0 = x_0 \sim \mu_0,\\
    \ud Y_t &= - \partial_x \hat H(t, X_t, Y_t, \mu_t) \ud t + Z_t \ud W_t + Z_t^0 \ud B_t, \quad Y_T = \partial_x g(X_T, \mu_T).
\end{aligned}
\end{equation}
Moreover, the optimal control is given by 
\begin{equation}
    \hat \alpha_t = \hat \alpha(t, X_t, Y_t, \mu_t),
\end{equation}
 for any solution $(X_t, Y_t, Z_t, Z_t^0)$ to FBSDE \eqref{def:FBSDE}.

The next theorem describes the McKean-Vlasov FBSDE for finding the mean-field equilibrium ({\it cf.} Definition~\ref{defn:MFG}).
\begin{thm}[Theorem 2.2.8, \citet{ahuja2015mean}]
Under Assumptions {\bf A1-A3}, the mean-field equilibrium of \eqref{def:J}-\eqref{def:Xt} exists if and only if the following McKean-Vlasov FBSDE is solvable:
\begin{equation}\label{def:MK-FBSDE}
\begin{aligned}
    \ud X_t &= b(t, X_t, \mc{L}(X_t \vert \mcF_t^B), \hat\alpha(t, X_t, Y_t, \mu_t))\ud t + \sigma(t, X_t,\mc{L}(X_t \vert \mcF_t^B)) \ud W_t + \sigma^0(t, X_t,\mc{L}(X_t \vert \mcF_t^B)) \ud B_t,\\
    \ud Y_t &= - \partial_x \hat H(t, X_t, Y_t, \mc{L}(X_t \vert \mcF_t^B)) \ud t + Z_t \ud W_t + Z_t^0 \ud B_t. %\quad Y_T = \partial_x g(X_T,\mc{L}(X_T \vert \mcF_T^B)).
\end{aligned}
\end{equation}
Moreover, the mean-field control-distribution flow pair is given by
\begin{equation}\label{def:MK-FBSDEequi}
    \alpha_t^\ast = \hat\alpha(t, X_t, Y_t, \mc{L}(X_t \vert \mcF_t^B)),  \quad \mu_t^\ast = \mc{L}(X_t \vert \mcF_t^B), \quad \forall t \in [0,T]. 
\end{equation}
\end{thm}
%To obtain mean-field equilibrium \eqref{defn:MFG}, it requires the following consistency for the above FBSDE system:
%\begin{equation}
%    \mu_t = \mc{L}(X_t^{\hat \alpha} \vert \mcF_t^B)
%\end{equation}
%where $X_t^{\hat\alpha}$ is the optimal controlled process for the problem \eqref{def:J}-\eqref{def:Xt} given $\mu$. 

\begin{thm}
Under Assumption~\ref{assump:cvg}, the FBSDE systems \eqref{def:FBSDE} and \eqref{def:MK-FBSDE} have unique solutions. Moreover, let $\mu_t^1, \mu_t^2 \in \mc{M}([0,T]; \mcP^2(\RR^d))$ be different given flow of measures, and denote by $(X_t^i, Y_t^i, Z_t^i, Z_t^{0,i})$ the unique solution to FBSDE \eqref{def:FBSDE} given $\mu_t^i$, then 
\begin{equation}\label{eq:estFBSDE}
    \EE\left[\sup_{t \in [0,T]} \| \Delta X_t\|^2 + \sup_{t \in [0,T]} \|\Delta Y_t\|^2 + \int_0^T \|\Delta Z_t\|^2 + \|\Delta Z_t^0\|^2 \ud t \right] \leq C_{K,T} \EE\left[\int_0^T (\Delta \mu_t)^2 \ud t \right],
\end{equation}
where $\Delta X_t = X_t^1 - X_t^2$, $\Delta Y_t, \Delta Z_t, \Delta Z_t^0$ are defined similarly, and $\Delta \mu_t = \mc{W}_2(\mu_t^1, \mu_t^2)$.
\end{thm}
\begin{proof}
The results generalize Theorem 3.1.3, Proposition 3.1.4 and Theorem 3.1.6 in \citet{ahuja2015mean} to the multi-dimensional case and  with Lipschitz SDE coefficients $b, \sigma, \sigma^0$. The original proofs rely on Theorem 3.1.1 and Theorem 3.1.2 under Assumption~{\bf H} in \citet{ahuja2015mean}. With the additional conditions on $(b, \sigma, \sigma^0)$ in our setting, Assumption~{\bf H} of \citet{ahuja2015mean} still holds. We omit the details because they essentially parallel the corresponding derivations in \citet{ahuja2015mean}.
\end{proof}

Now we are ready to prove Theorem~\ref{thm:cvg}.

\begin{proof}[Proof of Theorem~\ref{thm:cvg}]
The proof uses the estimate \eqref{eq:estFBSDE} repeatedly. We first observe that, for $\mu_t = \mc{L}(X_t \vert \mcF_t^B)$ and $\mu_t' = \mc{L}(X_t' \vert \mcF_t^B)$, one has
\begin{equation}\label{eq:estW2}
\EE[\mc{W}_2^2(\mu_t, \mu_t')] \leq \EE[\|X_t - X_t'\|^2], \quad \forall t \in [0,T].
\end{equation}
Then we define a map $\Phi$ by
\begin{equation}
    \mu = \{\mu_t\}_{0 \leq t \leq T} \to \Phi(\mu) := \{\mc{L}(X^\mu_t \vert \mcF_t^B)\}_{0 \leq t \leq T},
\end{equation}
where $X_t^\mu$ is the optimal controlled process in FBSDE \eqref{def:FBSDE} given $\mu \in \mc{M}([0,T]; \mcP^2(\RR^d))$. Combining \eqref{eq:estW2} and \eqref{eq:estFBSDE} gives
\begin{multline}\label{eq:mapest}
      \sup_{t \in [0,T]} \EE[\mc{W}_2^2(\Phi(\mu_t), \Phi(\mu_t'))] \leq \sup_{t \in [0,T]} \EE[\|X_t^\mu - X_t^{\mu'}\|^2] \\ \leq C_{K,T} \EE\left[\int_0^T \mc{W}_2^2(\mu_t, \mu_t') \ud t\right] \leq C_{K, T} T \sup_{t \in [0,T]} \EE[\mc{W}_2^2(\mu_t, \mu_t')].
\end{multline}
Thus, for sufficiently small $T$, $\Phi$ is a contraction map. By definition, $\mu_t^\ast$ defined in \eqref{def:MK-FBSDEequi} is a fixed point of $\Phi$: $\Phi(\mu^\ast) = \mu^\ast$. Let $\mu^{(0)}$ be the initial guess of $\mu^\ast$, and $\mu^{(n)}$ be the resulted flow of measures of $X_t$ given $\tilde \mu^{(n-1)}$ which is the approximation of $\mu^{(n-1)}$ by truncated signatures. So the measure flows are generated by
\begin{equation}
    \mu^{(0)} \to  \mu^{(1)}  \leadsto  \tilde \mu^{(1)} \to   \mu^{(2)}  \leadsto  \tilde \mu^{(2)} \cdots \to  \mu^{(n-1)}  \leadsto  \tilde \mu^{(n-1)} \to   \mu^{(n)}  \leadsto  \tilde \mu^{(n)}
\end{equation}
where $\to$ corresponds to the map $\Phi$, and $\leadsto$ corresponds to the truncated signature approximation. Therefore, with \eqref{eq:mapest} and the assumption $\sup_{t \in [0,T]} \EE[\mc{W}_2^2(\tilde{\mu}^{(n)}_t, \mu^{(n)}_t)] \leq \eps$ in Theorem~\ref{thm:cvg}, and denoting by $2C_{K,T}T = q$, we deduce that
\begin{align*}
      \sup_{t \in [0,T]} \EE[\mc{W}_2^2(\tilde \mu_t^{(n)}, \mu_t^\ast)] &\leq   2\sup_{t \in [0,T]} \EE[\mc{W}_2^2(\tilde \mu_t^{(n)}, \mu_t^{(n)})] +  2\sup_{t \in [0,T]} \EE[\mc{W}_2^2( \mu_t^{(n)}, \mu_t^\ast)] \\
      & \leq 2\eps + 2C_{K,T} T  \sup_{t \in [0,T]}\EE[\mc{W}_2^2( \tilde\mu_t^{(n-1)}, \mu_t^\ast)]  = 2\eps + q \sup_{t \in [0,T]}\EE[\mc{W}_2^2( \tilde\mu_t^{(n-1)}, \mu_t^\ast)] \\
      & \leq 2\eps + q (2\eps + q  \sup_{t \in [0,T]}\EE[\mc{W}_2^2( \tilde\mu_t^{(n-2)}, \mu_t^\ast)]) \\
      & \leq \cdots \\
      & \leq 2\eps (1 + q + q^2 + \ldots q^{n-1}) + q^n \sup_{t \in [0,T]}\EE[\mc{W}_2^2( \mu_t^{(0)}, \mu_t^\ast)] \\
      & = \frac{2-2q^n}{1-q}\eps + q^n\sup_{t \in [0,T]}\EE[\mc{W}_2^2( \mu_t^{(0)}, \mu_t^\ast)]. 
\end{align*}
With sufficiently small $T$, one has $0 < q < 1$. To estimate  $\int_0^T \EE|\alpha_t^{(n)} - \alpha_t^\ast|^2 \ud t$, we observe that
\begin{equation}
\alpha^{(n)}_t - \alpha^\ast_t = \hat \alpha(t, X_t^{\tilde \mu^{(n-1)}}, Y_t^{\tilde \mu^{(n-1)}}, \tilde \mu_t^{(n-1)}) - \hat \alpha(t, X_t^\ast, Y_t^\ast, \mu_t^\ast),
\end{equation}
where $ (X_t^{\tilde \mu^{(n-1)}}, Y_t^{\tilde \mu^{(n-1)}})$ is the solution to FBSDE \eqref{def:FBSDE} given $ \tilde \mu^{(n-1)}$, and $(X_t^\ast, Y_t^\ast)$ can be viewed as the solution to FBSDE \eqref{def:FBSDE} given $\mu^\ast$. Then using  the Lipschitz property of $\hat \alpha$ in $(t, x, \mu)$ and \eqref{eq:estFBSDE} again produces
\begin{align*}
  \int_0^T \EE|\alpha_t^{(n)} - \alpha_t^\ast|^2 \ud t 
  &\leq C_{K,T}\EE\left[\int_0^T \|X_t^{\tilde \mu^{(n-1)}} - X_t^\ast\|^2 + \|Y_t^{\tilde \mu^{(n-1)}} - Y_t^\ast\|^2 + \mc{W}_2^2( \tilde \mu_t^{(n-1)}, \mu_t^\ast )  \ud t \right]  \\
  & \leq C_{K,T} T \sup_{t \in [0,T]} \EE[\mc{W}_2^2( \tilde \mu_t^{(n-1)}, \mu_t^\ast )].
\end{align*}
Therefore, we obtain the desired result.
\end{proof}

\vspace{-0.6em}
Next we give the proof to Theorem~\ref{thm:numcvg}. 
\vspace{0em}
\begin{proof}[Proof of Theorem~\ref{thm:numcvg}]
Consider a partition of $[0,T]: 0=t_0<\cdots<t_L=T$, and define $\pi(t) = t_k$ for $t \in [t_k, t_{k+1})$ with $\|\pi\| = \max_{1 \leq k < L}|t_{k} - t_{k-1}|$, then by following the line of the proof to Theorem~\ref{thm:cvg}, one only needs an additional estimate on $\EE|X_t^\mu - X_{t_k}^{(n)}|^2$ to complete the proof. Noticing that $X_t$ solves \eqref{def:Xt} with $\mu^\ast$ and $X_{t_k}^{(n)}$ satisfies \eqref{def:Xt_discrete} with $\tilde \mu^{(n-1)}$, one can obtain the estimate by following Lemma 14 in \citet{carmona2019convergence} with $N = 1$. 
%We essentially follow the proof of Theorem 3 in \citet{han2020convergence}, and thus omit the detials here.
\end{proof}

\section{Benchmark Solutions}\label{app:Benchmark}
This appendix summarizes the analytical solutions to the three examples in Section~\ref{sec:numerics}, which are used to benchmark our algorithm's performance.

\paragraph{Linear-Quadratic MFGs.}
The analytical solution is provided in \citet{carmona2015mean}:
\begin{align}
    & m_t := \EE[X_t \vert \mcF_t^B]= \EE[X_0] + \rho\sigma B_t, \quad t\in[0,T], \label{def:LQ_m} \\
    & \alpha_t = (q+\eta_t) (m_t -X_t), 
    \quad t\in[0,T], \label{def:LQ_alpha}
\end{align}
where $\eta_t$ is a deterministic function solving the Riccati equation:
\begin{equation*}
    \dot \eta_t = 2(a+q)\eta_t + \eta_t^2 - (\eps - q^2), \quad \eta_T = c,
\end{equation*}
with the solution given by
\begin{equation*}
    \eta_t = \frac{-(\epsilon-q^2)(e^{(\delta^+-\delta^-)(T-t)}-1) -c(\delta^+e^{(\delta^+-\delta^-)(T-t)}-\delta^-)}{(\delta^-e^{(\delta^+-\delta^-)(T-t)} - \delta^+) -c(e^{(\delta^+-\delta^-)(T-t)} -1)}.
\end{equation*}
Here $\delta^\pm = -(a+q)\pm \sqrt{R}$, $R = (a+q)^2 + (\epsilon-q^2)>0$, and the minimized expected cost is $V(0, x_0 - \EE[x_0])$ with 
\begin{equation*}
    V(t,x) = \frac{\eta_t}{2}x^2 + \mu_t, \quad \mu_t = \half \sigma^2(1-\rho^2)\int_t^T \eta_s \ud s.
\end{equation*}

The benchmark trajectories in Figure~\ref{fig:LQ} are simulated according to \eqref{def:LQ_SDE}  with $m_t$ and $\alpha_t$ in \eqref{def:LQ_m} and \eqref{def:LQ_alpha}.

\paragraph{Mean-field Portfolio Game} Given the type vector $\zeta=(\xi, \delta, \theta, \mu, \nu, \sigma)$, the analytical solution provided in  \citet{lacker2018mean} is summarized below
\begin{align*}
    & \pi^*_t = \delta\frac{\mu}{\sigma^2+\nu^2} + \theta\frac{\sigma}{\sigma^2+\nu^2}\frac{\phi}{1-\psi}, \\
    & m_t = \EE[\xi] +\EE[\mu \pi^*]t + \EE[\sigma \pi^*]B_t,
\end{align*}
where $\phi=\EE[\delta\frac{\mu\sigma}{\sigma^2+\nu^2}]$ and $\psi = \EE[\theta\frac{\sigma^2}{\sigma^2+\nu^2}]$.
Note that, since the type vector $\zeta$ is random representing the heterogenuity of agents in this mean-field game, $\pi^*$ is a random strategy. The maximized expected utility of this game is given by $\EE[v(0, \xi - \theta \EE[\xi])]$, with
%where $v(t,x)$ is given by
\begin{equation*}
    v(t,x) = -e^{-x/\delta}e^{-\rho(T-t)}, \quad \rho = \frac{1}{2(\sigma^2 + \nu^2)} \left(\mu + \frac{\theta}{\delta}\frac{\phi}{1-\psi}\sigma\right)^2  - \frac{\theta}{\delta}\left(\tilde\psi + \frac{\tilde \phi \phi}{1-\psi}\right) - \half\left(\frac{\theta}{\delta}\frac{\phi}{1-\psi}\right)^2,
\end{equation*}
%with 
\begin{equation*}
     \quad \tilde\psi = \EE\left[\delta \frac{\mu^2}{\sigma^2 + \nu^2}\right], \quad \tilde\phi = \EE\left[\theta \frac{\mu\sigma}{\sigma^2 + \nu^2}\right].
\end{equation*}
Note that Figure~\ref{fig:Invest}(c) plots the absolute value of $\EE[v(0, \xi - \theta \EE[\xi])]$.

\paragraph{Mean-field Game of Optimal Consumption and Investment}

Following \citet{lacker2020many}, the analytical solution is given by
\begin{align}
    &\pi^*_t \equiv \pi^\ast = \frac{\delta\mu}{\sigma^2+\nu^2} - \frac{\theta(\delta-1)\sigma}{\sigma^2+\nu^2}\frac{\phi}{1+\psi}, \quad c^*_t = \left(\frac{1}{\beta} + (\frac{1}{\lambda}-\frac{1}{\beta})e^{-\beta(T-t)} \right)^{-1}, \label{def: InvestConsump c}
\end{align}
%where $\phi$, $\psi$, $\lambda$, $\beta$ are given in \cite{lacker2020many}. 
where 
\begin{align*}
    &\phi = \EE\left[\frac{\delta\mu\sigma}{\sigma^2+\nu^2}\right], \quad 
    \psi = \EE\left[\frac{\theta(\delta-1)\sigma^2}{\sigma^2+\nu^2}\right],\quad
    \lambda = \epsilon^{-\delta}\,\left(e^{ \EE\left[\log(\epsilon^{-\delta})\right]}\right)^{-\frac{\theta(\delta-1)}{1+\EE\left[\theta(\delta-1)\right]}},\\
    &\beta = \theta(\delta-1)\frac{\EE\left[\delta\rho\right]}{1+\EE\left[\theta(\delta-1)\right]}-\delta\rho,
\end{align*}
and 
\begin{align*}
    \rho = \left(1-\ovd\right)\,&\left\{\frac{\delta}{2(\sigma^2+\nu^2)}\left(\mu-\sigma\frac{\phi}{1+\psi}\theta(1-\ovd)\right)^2
    +\frac{1}{2}\left(\frac{\phi}{1+\psi}\right)^2\theta^2\left(1-\ovd\right)\right.\nonumber\\
    &-\theta\EE\left[\frac{\delta\mu^2-\theta(\delta-1)\sigma\mu\frac{\phi}{1+\psi}}{\sigma^2+\nu^2}\right]
    +\left. \frac{\theta}{2}\EE\left[\frac{(\delta\mu-\theta(\delta-1)\sigma\frac{\phi}{1+\psi})^2}{\sigma^2+\nu^2}\right]\right\}.
\end{align*}
    
Note that the expression of $m_t$, $\Gamma_t$ and the maximized expected utility are not given in \citet{lacker2020many}. For completeness,  we give their derivations below. Since $c^*_t$ in \eqref{def: InvestConsump c} doesn't depend on the common noise $B$, $\Gamma_t := \exp \EE[\log c_t^\ast \vert \mcF_t^B]$ admits a unique formula for all agents
\begin{equation*}
    \Gamma_t = \exp \EE[\log c^*_t].
\end{equation*}
To obtain the formula for $m_t := \exp\EE[\log X_t^\ast \vert \mcF_t^B]$, we first deduce by It\^{o}'s formula that
\begin{equation}
    \ud \log X_t^\ast = \pi_t^\ast (\mu \ud t + \nu \ud W_t + \sigma \ud B_t) - \frac{1}{2}  (2c_t^* + (\pi_t^\ast)^2\sigma^2 + (\pi_t^\ast)^2\nu^2 )\ud t, \label{def: InvestConsumplogSDE}
\end{equation}
from which we easily get 
$$
\EE[\log X_t^\ast|\mcF^B_t] = \EE[\log \xi]+ \EE[\pi^*\mu - \frac{1}{2}(\pi^*)^2(\sigma^2+\nu^2)] t - \int_0^t \EE[c^*_s] \ud t + \pi^* \sigma B_t,
$$
and $m_t = \exp \EE[\log X_t^\ast|\mcF^B_t]$. The maximized expected utility of this game is given by $\EE[v(0, \xi, \EE[\xi])]$, with 
\begin{equation*}
v(t,x,y) = \eps\left(1-\ovd\right)^{-1} x^{1-\ovd}y^{-\theta(1-\ovd)}f(t),
\end{equation*}
and $f(t)$ is defined by
\begin{equation*}
    f(t) = \exp\left\{\int_t^T \left(\rho + \ovd c_s^* + \EE[c_s^*]\left(1-\ovd\right)\theta\right) \ud s \right\}.
\end{equation*}
Note that, to ensure the positiveness of $X_t$ required by using the power utility, the trajectories of $X_t$ are obtained by simulating $\log X_t$ via \eqref{def: InvestConsumplogSDE} then taking the exponential. 

%We will use an equivalent expression of \eqref{def: InvestConsumpSDE} derived by It\^{o}'s formula to ensure the sign of $X_t$,
%\begin{equation}
%    \ud \log X_t = \pi_t (\mu \ud t + \nu \ud W_t + \sigma \ud B_t) - \frac{1}{2}  (2c_t + \pi^2_t\sigma^2 + \pi^2_t\nu^2 )\ud t, 
%\end{equation}
%with $X_t, c_t>0 \text{ and } X_0=\xi$.

\section{Plots of $\pi_t$, $c_t$, $\Gamma_t=\exp\EE(\log c_t |\mcF^B_t)$ for Mean-Field Game of Optimal Consumption and Investment}\label{app:OCI}

\begin{figure*}[ht]
    \centering
     \subfloat[$\pi_t$]{
         \includegraphics[width=0.32\columnwidth]{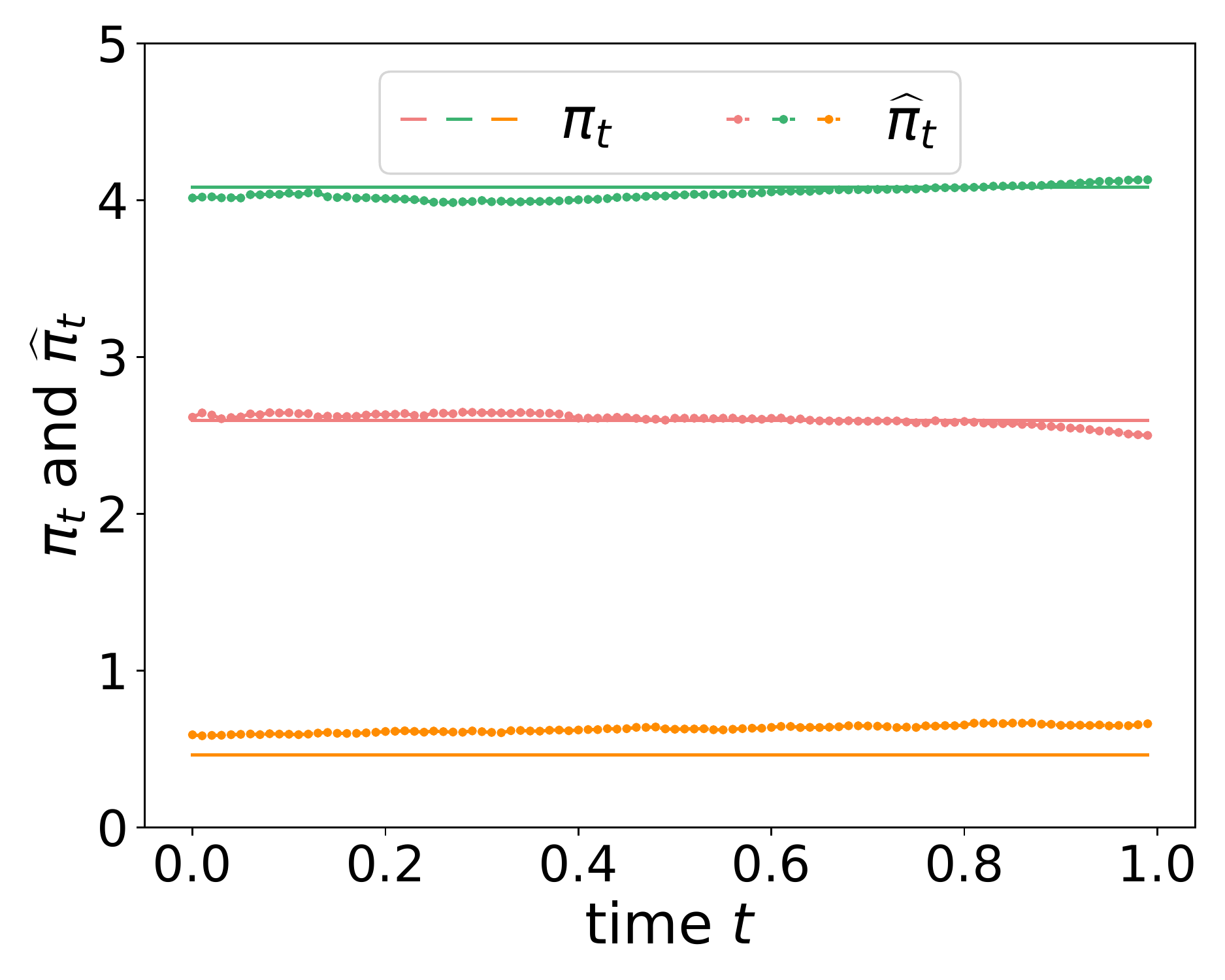}} 
    \subfloat[$c_t$]{
         \includegraphics[width=0.32\columnwidth]{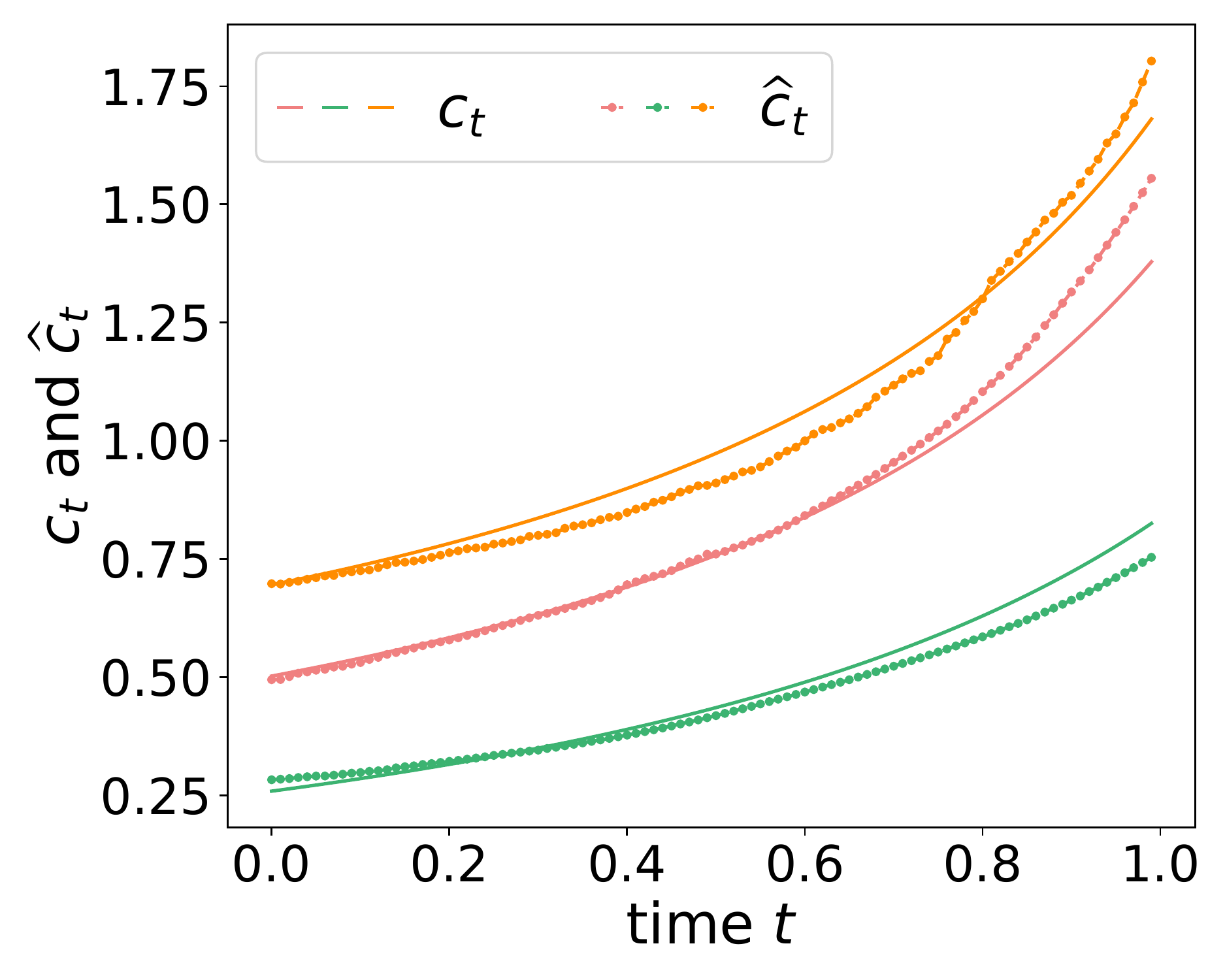}} 
     \subfloat[$\Gamma_t =  \exp\EE(\log c_t |\mcF^B_t)$]{
         \includegraphics[width=0.32\columnwidth]{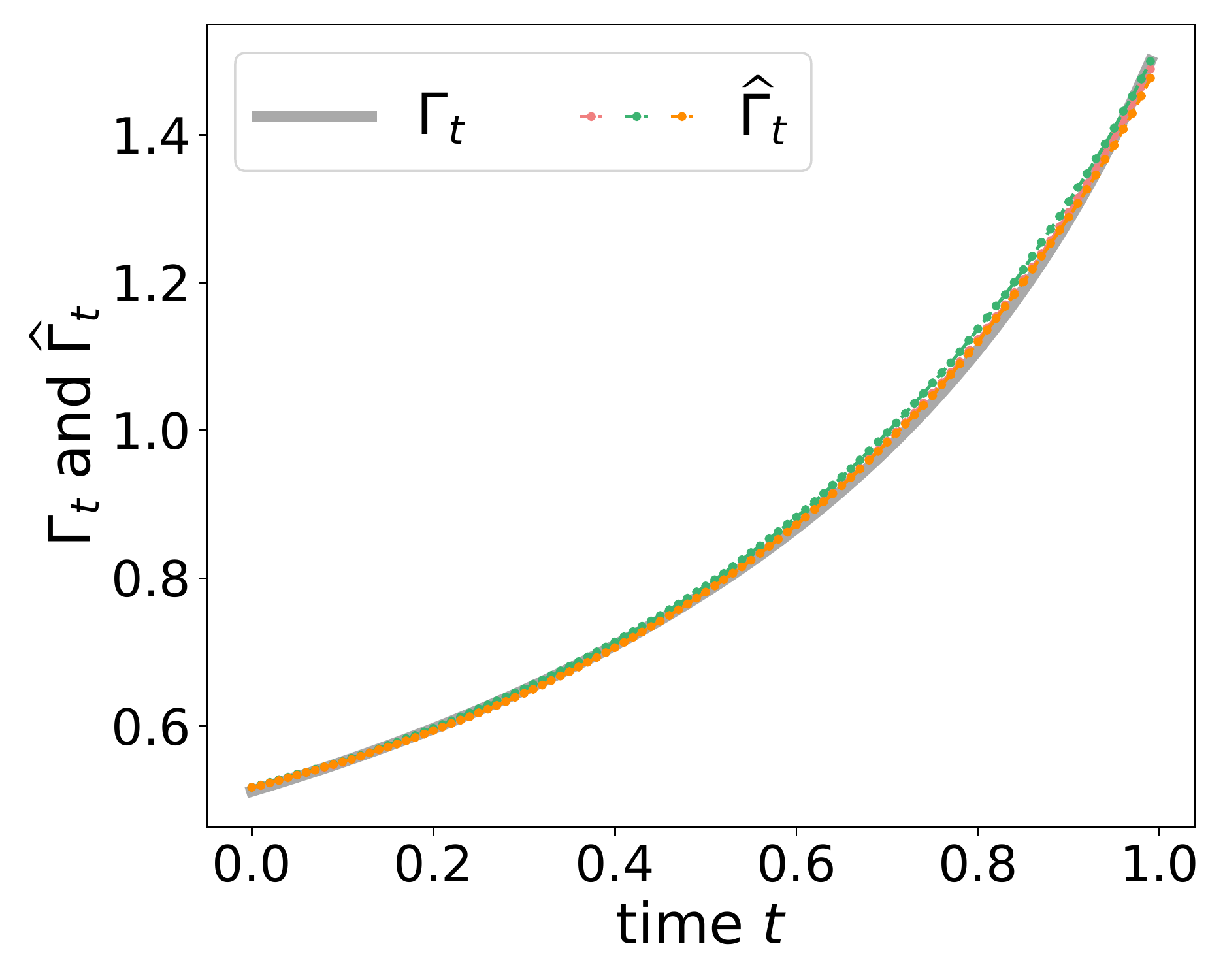}}
    \caption{Plots on test data for three different $(X_0^i, W^i, B^i, \zeta^i)$. Solid line is the benchmark solution and dashed line is the numerical approximation using the Sig-DFP algorithm. Each panel presents three trajectories of $\pi_t$, $c_t$, and $\Gamma_t=\exp\EE(\log c_t |\mcF^B_t)$ and their approximations. Parameter choices are: $\delta \sim U(2, 2.5), \mu\sim U(0.25, 0.35), \nu\sim U(0.2, 0.4), \theta, \xi \sim U(0,1), \sigma\sim U(0.2, 0.4)$, $\epsilon\sim U(0.5, 1)$.}\label{fig:OCIcurves}
\end{figure*}

\section{Experiment setup for the high-dimensional case $n_0=5$}\label{app:highd}

To test the performance of Sig-DFP in high dimensions, we implement a toy experiment on the mean-field game of optimal consumption and investment with the common noise of dimension $n_0=5$. Specifically, we modify the $\sigma \ud B_t$ term in \eqref{def: InvestConsumpSDE} to be in high dimensions, {\it i.e.}, $X_t$ now follows
\begin{equation*}
    \ud X_t = \pi_t X_t(\mu \ud t + \nu \ud W_t + {\bm\sigma}\transpose \ud {\bm B}_t) - c_tX_t \ud t,  
\end{equation*}
where ${\bm\sigma} := (\sigma_1, \dots, \sigma_5)\transpose$, ${\bm B}_t$ is a $5$-dimensional Brownian motion, and $X_0 = \xi$. We use the same hyperparameters for training and provide the running time in Table~\ref{tab:depth}. %We remark that, unlike the $n_0=1$ case, one does not have closed-form solutions as a benchmark to compare with for $n_0=5$.

\end{onecolumn}

%\section{Do \emph{not} have an appendix here}

%\textbf{\emph{Do not put content after the references.}}
%
%Put anything that you might normally include after the references in a separate
%supplementary file.

%We recommend that you build supplementary material in a separate document.
%If you must create one PDF and cut it up, please be careful to use a tool that
%doesn't alter the margins, and that doesn't aggressively rewrite the PDF file.
%pdftk usually works fine. 

%\textbf{Please do not use Apple's preview to cut off supplementary material.} In
%previous years it has altered margins, and created headaches at the camera-ready
%stage. 
%%%%%%%%%%%%%%%%%%%%%%%%%%%%%%%%%%%%%%%%%%%%%%%%%%%%%%%%%%%%%%%%%%%%%%%%%%%%%%%
%%%%%%%%%%%%%%%%%%%%%%%%%%%%%%%%%%%%%%%%%%%%%%%%%%%%%%%%%%%%%%%%%%%%%%%%%%%%%%%

\end{document}